\documentclass[11pt]{amsart}

\usepackage{subfiles}
\usepackage{lipsum}
\usepackage[font=itshape]{quoting}
\usepackage{tabu, fullpage,diagbox, booktabs}
\usepackage{placeins}
\usepackage[margin=0.8in]{geometry}
\usepackage[dvipsnames]{xcolor}   		             		
\usepackage{graphicx}			
\usepackage{amssymb,csquotes}
\usepackage{dsfont,comment}
\usepackage{mathrsfs}
\usepackage{amsthm}
\usepackage{amsmath}
\usepackage{stmaryrd}
\usepackage{tikz}
\usepackage{tikz-cd}
\usepackage{accents}
\usepackage{upgreek}
\usepackage{enumerate}
\usepackage{bm}
\usepackage{mathtools}
\usepackage[all]{xy}
\usepackage{caption}

\usepackage{url}
\usepackage{float}
\usepackage{todonotes}
\usepackage{bbm}
\usepackage{easyeqn}

\usepackage[full]{textcomp}
\usepackage[cal=cm]{mathalfa}
\usepackage{xparse}
\usepackage{comment}
\usepackage{cite}

\usepackage{setspace}
\tikzset{
	commutative diagrams/.cd, 
	arrow style=tikz, 
	diagrams={>=stealth}
}

\theoremstyle{definition}
\newtheorem{innercustomthm}{Theorem}

\newenvironment{customthm}[1]
{\renewcommand\theinnercustomthm{#1}\innercustomthm}
{\endinnercustomthm}

\makeatletter
\newcommand{\colim@}[2]{%
	\vtop{\m@th\ialign{##\cr
			\hfil$#1\operator@font colim$\hfil\cr
			\noalign{\nointerlineskip\kern1.5\ex@}#2\cr
			\noalign{\nointerlineskip\kern-\ex@}\cr}}%
}
\newcommand{\colim}{%
	\mathop{\mathpalette\colim@{\rightarrowfill@\textstyle}}\nmlimits@
}
\makeatother

\makeatletter
\def\@tocline#1#2#3#4#5#6#7{\relax
	\ifnum #1>\c@tocdepth 
	\else
	\par \addpenalty\@secpenalty\addvspace{#2}%
	\begingroup \hyphenpenalty\@M
	\@ifempty{#4}{%
		\@tempdima\csname r@tocindent\number#1\endcsname\relax
	}{%
		\@tempdima#4\relax
	}%
	\parindent\z@ \leftskip#3\relax \advance\leftskip\@tempdima\relax
	\rightskip\@pnumwidth plus4em \parfillskip-\@pnumwidth
	#5\leavevmode\hskip-\@tempdima
	\ifcase #1
	\or\or \hskip 1em \or \hskip 2em \else \hskip 3em \fi%
	#6\nobreak\relax
	\dotfill\hbox to\@pnumwidth{\@tocpagenum{#7}}\par
	\nobreak
	\endgroup
	\fi}
\makeatother

\usetikzlibrary{calc}
\usetikzlibrary{fadings}
\usetikzlibrary{decorations.pathmorphing}
\usetikzlibrary{decorations.pathreplacing}

\newcounter{marginnote}
\DeclareMathAlphabet{\mathpzc}{OT1}{pzc}{m}{it}

\usepackage[backref=page]{hyperref}
\hypersetup{
	colorlinks   = true,          
	urlcolor     = blue,          
	linkcolor    = violet,          
	citecolor   = blue             
}

\theoremstyle{definition}

\newtheorem{theorem}{Theorem}[subsection]

\newtheorem{corollary}[theorem]{Corollary}
\newtheorem{lemma}[theorem]{Lemma}
\newtheorem{proposition}[theorem]{Proposition}

\newtheorem{remark}[theorem]{Remark}

\newtheorem*{runningexample*}{Running example}

\newtheorem{constr}[theorem]{Construction}

\newtheorem{definition}[theorem]{Definition}

\newtheorem{example}[theorem]{Example}

\newtheorem{proposition-definition}[theorem]{Proposition-Definition}

\newenvironment{construction}    
{%
	\pushQED{\qed}\begin{constr}}
	{\popQED\end{constr}}

\newcommand{\RR}{\mathbb{R}}

\newcommand{\bcd}{\begin{center}\begin{tikzcd}}
	\newcommand{\ecd}{\end{tikzcd}\end{center}}

\newcommand{\A}{\mathbb{A}}

\newcommand{\cG}{\mathcal{G}}
\newcommand{\cY}{\mathcal{Y}}

\newcommand{\Spec}{\operatorname{Spec}}

\usepackage{accents}

\DeclareMathAlphabet{\mathpzc}{OT1}{pzc}{m}{it}

\NewDocumentCommand{\compatibilitydatum}{m m m m m m O{} O{} O{}}{
	\begin{equation*} \begin{tikzcd}[ampersand replacement=\&]
	\: \arrow{r} \& {#1} \arrow{r} \arrow{d}{#7} \& {#2} \arrow{r} \arrow{d}{#8} \& {#3} \arrow{r}{[1]} \arrow{d}{#9} \& \: \\
	\: \arrow{r} \& {#4} \arrow{r} \& {#5} \arrow{r} \& {#6} \arrow{r} \& \:
	\end{tikzcd} \end{equation*}}

\NewDocumentCommand{\commutingsquare}{m m m m o O{} O{} O{} O{}}{
	\begin{equation}\begin{tikzcd}[ampersand replacement=\&] \label{#5}
	#1 \arrow{r}{#6} \arrow{d}{#7} \& #2 \arrow{d}{#8} \\
	#3 \arrow{r}{#9} \& #4
	\end{tikzcd}\IfValueTF{#5}{\label{#5}}{} \end{equation}}

\NewDocumentCommand{\cartesiansquarelabel}{m m m m m O{} O{} O{} O{}}{
	\begin{tikzcd}[ampersand replacement=\&]
	#1 \arrow{r}{#6} \arrow{d}{#7} \arrow[dr, phantom, "\square"] \& #2 \arrow{d}{#8} \\
	#3 \arrow{r}{#9} \& #4
	\end{tikzcd}\IfValueTF{#5}{\label{#5}}{}
}

\NewDocumentCommand{\triangleofspaces}{m m m O{} O{} O{}}{
	\begin{tikzcd} [ampersand replacement=\&]
	#1 \arrow{r}{#4} \arrow[bend right]{rr}{#5} \& #2 \arrow{r}{#6} \& #3
	\end{tikzcd}}

 \newcommand{\cal}{\mathcal}

\def\cY{{\cal Y}}

\begin{document}
	\title{Logarithmic Quot spaces, boundedness, and $K$-tropicalizations}
	\author{Patrick Kennedy--Hunt {\it\&} Dhruv Ranganathan}

	\begin{abstract}
	Logarithmic Hilbert and Quot spaces are generalizations of their traditional versions adapted to study pairs and degenerations. The logarithmic Quot spaces of $(X,D)$ parameterize ``algebraically transverse'' (logarithmically flat) quotient sheaves on degenerations of $X$. We prove boundedness and deduce properness of logarithmic Quot spaces. The results complete the basic foundations of logarithmic Quot spaces and specialize to work of Li--Wu and Maulik and the second author in special cases. 

    Boundedness relies on two results of independent interest.  First, we show that for a simple normal crossing pair $(X, D)$ and a subscheme $Z$, there is a smallest logarithmic space ${X}^\flat$ modifying $X$ such that the strict transform of $Z$ is algebraically transverse. Precisely, given $Z\hookrightarrow X$, there is a canonical logarithmic space $X^\flat$ over $X$ with the following universal property -- an snc logarithmic blowup $X'\to X$ makes the strict transform of $Z$ algebraically transverse if and only if $X'$ is a modification of $X^\flat$. Parallel results hold for arbitrary coherent sheaves. This proves boundedness for logarithmic quotients with fixed tropicalization. 

    A logarithmic quotient sheaf defines a $K$-tropicalization, an enhancement of tropicalization that is sensitive to scheme  structures. The $K$-tropicalization has the same relationship to $K$-theory as traditional tropicalization has to Chow, and is related to Gr\"obner theory and convex geometry via state and secondary polytopes.
    Using the $K$-theory of toric bundles, we derive a balancing condition for $K$-tropicalizations that imposes strong finiteness properties. The second key result is that $K$-tropicalizations with fixed numerics are parametrized by a finite-dimensional polyhedral complex. 
	\end{abstract}

    \address{Patrick Kennedy-Hunt and Dhruv Ranganathan \\ Department of Pure Mathematics {\it \&} Mathematical Statistics\\
University of Cambridge, Cambridge, UK}
\email{\href{mailto:pfk21@cam.ac.uk}{pfk21@cam.ac.uk}}
\email{\href{mailto:dr508@cam.ac.uk}{dr508@cam.ac.uk}}

	\maketitle 
	\tableofcontents
 
\section*{Introduction}

\subsection{Context} The logarithmic Hilbert and Quot spaces, recently introduced by the first author, are generalizations of the ordinary Hilbert and Quot schemes that are well-suited to studying moduli of subschemes and quotients in degenerations, and more generally, on simple normal crossings pairs~\cite{Logquot}. Simple instances of the logarithmic Hilbert scheme, namely for subschemes of dimension $0$ and $1$, are basic objects in Donaldson--Thomas theory~\cite{LiWu,MR20,MR23} and have been used to striking effect, for instance in~\cite{PP17}. 

The purpose of this paper is to prove boundedness, and therefore properness, of the logarithmic Quot spaces in full generality, completing the foundational theory initiated in~\cite{Logquot}.

We start with an overview of the basic geometry of these spaces, and focus on the Hilbert space for simplicity. Let $(X,D)$ be a simple normal crossings pair. The simplest objects in the logarithmic Hilbert scheme $\mathsf{Hilb}(X|D)$ are subschemes
\[
Z\hookrightarrow X
\]
that are {\it algebraically transverse}, meaning that the functions cutting out the components of $D$ form regular sequences when pulled back to $Z$. Subschemes with this property are stable under blowups of $X$ along strata, in the sense that their strict and total transforms coincide. This makes them natural from the perspective of logarithmic geometry, where many natural constructions are expected to be invariant under strata blowups.  

The space of algebraically transverse subschemes is an open in the usual Hilbert scheme of $X$, and we denote it $\mathsf{Hilb}^\circ(X|D)$. The logarithmic Hilbert scheme is a larger moduli problem
\[
\mathsf{Hilb}^\circ(X|D)\hookrightarrow \mathsf{Hilb}(X|D)
\]
on the category of {\it logarithmic schemes}. The degenerate objects in $\mathsf{Hilb}(X|D)$ are subschemes
\[
Z\hookrightarrow \mathcal X
\]
of {\it expansions} of $X$ along $D$. An expansion is a reducible, reduced, normal crossings space, constructed by iterated deformation to the normal cone along the strata of $(X,D)$. In the expansion, algebraic transversality can be maintained component-by-component, and this provides a natural class of limits for families of algebraically transverse subschemes. The case of the Quot space is similar, see Section~\ref{sec: sheaves-on-expansions}. 

In recent work of the first author, these ideas are used to describe a logarithmic moduli problem parameterizing objects as above, subject to a certain equivalence relation. Basic properties of this moduli problem are shown there, including: (i) representability as a logarithmic space\footnote{A logarithmic space is a stack over logarithmic schemes that admits a logarithmically \'etale cover by a scheme. Informally, it becomes a stack with algebraic structure after a blowup. In our cases, there is a topological space associated to a logarithmic space, and each polyhedral complex structure on this space gives rise to a scheme theoretic model.  See~\cite{Logquot,MW23,MolchoWise} for examples and further discussion.}, (ii) separatedness, and (iii) universal closedness. It is also shown that each object in the moduli space has a tropicalization -- a piecewise linear stratification of the cone complex of $(X,D)$ that captures the complexity of the expansion needed to represent a subscheme. See~\cite{Logquot} and Section~\ref{sec: prelim} for further details.

\subsection{Main results} We show that the connected components of the logarithmic Quot schemes are of finite type, and conclude properness. Let $(X,D)$ be a simple normal crossings pair and let $\mathcal E$ be a coherent sheaf on $X$ that is logarithmically flat over the Artin fan of $(X|D)$. 

\begin{customthm}{A}[Boundedness]\label{thm: boundedness-quot-scheme}
Every connected component of $\mathsf{Quot}^{\sf log}(X|D,\mathcal E)$ is of finite type, and therefore proper. The space parameterizing quotients with fixed Hilbert polynomial with respect to an ample on $X$ is proper. 
\end{customthm}

The study of the logarithmic Quot scheme is part of a broader picture, to understand how moduli of stable sheaves behave under snc degenerations. While such a theory has not taken its final form, we expect the result above will form the basis for boundedness of these putative moduli problems. When $D$ is a smooth divisor, the result specializes to a well-known result of Li--Wu~\cite{LiWu}.

The theorem relies on two key results. The first is geometric. Given any subscheme $Z\hookrightarrow X$ of a pair $(X,D)$, there is a blow-up $X'\to X$ along strata such that the strict transform $Z'$ is algebraically transverse; for a construction of such transversalizations, see~\cite{Logquot,tevelev2005compactifications,TevNotes,TV21}. 

We show, in the category of logarithmic spaces, there is a unique smallest blowup of $(X,D)$ such that the strict transform of $\mathcal F$ is algebraically transverse, with the caveat that this blowup is a logarithmic space, rather than a scheme or algebraic stack. We prove a parallel result for coherent sheaves. 

The result is easiest to state in the language of cone complexes. An snc pair $(X,D)$ has a cone complex $\Sigma$, and iterated blowups along smooth strata correspond to iterated stellar refinements of $\Sigma$, see~\cite{KKMSD}.

\begin{customthm}{B}[Canonical transversalization]\label{thm: canonical-transversalization}
Let $(X,D)$ be a simple normal crossings pair with fan $\Sigma$ and let $\mathcal F$ be a coherent sheaf on $X$. There is a canonical piecewise linear stratification $\Sigma^\flat_{\mathcal F}$ such that, if 
\[
\Sigma'\to \Sigma
\]
is a smooth subdivision inducing $X'\to X$, then the strict transform $\mathcal F'$ of $\mathcal F$ is algebraically transverse if and only if $\Sigma'$ refines $\Sigma^\flat_{\mathcal F}$. 
\end{customthm}

Equivalently, there exists a logarithmic space $X^\flat_{\mathcal F}\to X$ such that all snc logarithmic blowups that algebraically transversalize $\mathcal F$ are blowups of $X^\flat$. 

The piecewise linear stratification in the theorem is geometrically meaningful. For each integral point of $\Sigma$, subdivide by introducing a ray through this point and record the intersection of the strict transform of $\mathcal F$ with the interior of the corresponding exceptional divisor -- a torus bundle over a stratum of $(X,D)$. We then stratify according to the sheaf on this torus bundle, see Section~\ref{sec: prelim}. The analogue of this result in setting when $D$ is smooth is trivial, and this result is one of the essential difficulties in going beyond the setting considered by Li--Wu~\cite{LiWu}.

The theorem is used to prove boundedness of the logarithmic Quot space, but the statement is purely about the coherent sheaf, not the presentation as a quotient. We therefore expect it to play a role in further study of logarithmic moduli of coherent sheaves. 

The result is in the spirit of the famous Raynaud--Gruson result on flattening by blow--ups~\cite{RG71}, and might be viewed as such a statement over the Artin fan of $X$. Again, it is crucial to work with logarithmic spaces that may not be representable by schemes, and so cannot follow from the classical statement alone. The subtlety seems to stem from the non-properness of the map $X\to\mathsf A(X,D)$.

Theorem~\ref{thm: canonical-transversalization} is a $K$-theoretic enhancement of a foundational cycle-theoretic result in tropical geometry. A subscheme $Z\hookrightarrow X$ is {\it dimensionally transverse} if intersections with strata of $X$ have expected dimension\footnote{Algebraic transversality implies dimensional transversality, but the converse does not hold in general.}. The usual tropicalization $Z^{\sf trop}$ is a piecewise linear stratification of $\Sigma$ such that for $X'\to X$ a blowup as above, the strict transform $Z'\hookrightarrow X'$ is dimensionally transverse if and only if $\Sigma'$ refines $Z^{\sf trop}$, see~\cite[Section~14]{Gub:guideToTrop}\footnote{In~\cite{Gub:guideToTrop} Gubler states and proves an equivalent formulation of this result and attributes it to Payne.} and~\cite{ulirsch2014tropical}. The construction depends only on the point $[Z]$ in the Chow variety of $X$, and information about non-reduced or embedded components is irrelevant. For subschemes or sheaves, the natural transversality is {\it algebraic} transversality, and Theorem~\ref{thm: canonical-transversalization} is the necessary strengthening.

A key property of traditional tropicalization is that the maximal strata are endowed with multiplicities and these multiplicities satisfy a balancing condition. The multiplicities are computed from the intersection theory of strata of $X'$ with $Z'$.  The stratification occurring in Theorem~\ref{thm: canonical-transversalization} is decorated by multiplicities on {\it all} strata, given by Euler characteristics or Hilbert polynomials. We call this decorated object a {\it $K$-tropicalization}, and every point of the logarithmic Quot scheme has a $K$-tropicalization. We establish a replacement for the balancing condition using $K$-theory and prove the following. 

\begin{customthm}{C}[$K$-tropicalizations are bounded]\label{thm: tropical-finiteness}
The set of combinatorial types $K$-tropicalizations that can arise from logarithmic quotient sheaves with prescribed numerical data is finite. 
\end{customthm}

The ``numerical data'' is locally constant data that arise from Hilbert polynomials of the sheaves restricted to strata. On a component of an expansion of $X$ along $D$, the polarizing line bundle on $X$ need not be ample. The Hilbert polynomial is therefore a weaker invariant than it is in the standard theory of Hilbert schemes. The theorem gives a combinatorial compensation for this weakening. 

Precisely, the result guarantees that only finitely many tropicalizations, and so expansions, arise from any logarithmic Quot scheme with fixed numerical data. The previous theorem guarantees that quotients with given $K$-tropicalization can be parameterized by finite type open subscheme an ordinary Quot schemes. This is sufficient to conclude boundedness.

When $X$ is toric with torus $T$, the finiteness result for $K$-tropicalization is closely related to some remarkable combinatorial objects at the intersection of commutative algebra and polyhedral combinatorics -- the Chow, state, and secondary polytopes~\cite{KSZ92,Sturmfels96}. Given a subscheme $Z\hookrightarrow X$ its {\it state polytope} is the polytope\footnote{The construction depends on a choice of ample on the Hilbert scheme, and there are, in truth, many state polytopes depending on the polarization.} of the $T$-orbit closure of $[Z]$ in the ordinary Hilbert scheme of $X$. Its dual fan is the {\it Gr\"obner fan} of $Z$. The strata in the $K$-tropicalization are unions of cells in the Gr\"obner fan, and space of $K$-tropicalizations with fixed numerics is closely related to the secondary polytope of the state polytope. 

\subsection{Overview} We highlight a few key ideas in the paper. 

\subsubsection{Sheaves with a fixed tropicalization} Let $\Gamma$ be a polyhedral decomposition of the cone complex $\Sigma$ of $(X,D)$. There is an associated expanded degeneration 
\[
p\colon X_\Gamma\to X.
\]
Fix $\mathcal E$ an algebraically flat coherent sheaf on $X$ and a quotient $[p^\star \mathcal E\twoheadrightarrow \mathcal F]$, with $\mathcal F$ algebraically transverse. 

The sheaf $\mathcal F$ has discrete invariants. Fix an ample $\mathcal O_X(1)$ on $X$. There is a natural flat specialization from $X$ to $X_\Gamma$, and each stratum of $X$, including $X$ itself, undergoes an expansion, and therefore gives ``broken stratum'', i.e. a union of strata, of $X_\Gamma$. The Hilbert polynomials of the restrictions of $\mathcal F$ to these broken strata, with respect to the pullback $p^\star \mathcal O(1)$ give numerical data for the sheaf.

If $\Gamma'\to \Gamma$ is a refinement of polyhedral decompositions, the pullback of $\mathcal F$ under
\[
X_{\Gamma'}\to X_\Gamma
\]
is another algebraically transverse sheaf on an expansion of $X$, and since the polarization is pulled back, the numerical data is unchanged. The natural equivalence relation on the logarithmic Quot scheme identifies these objects\footnote{This can be viewed as a kind of sheaffification in the topology of logarithmic modifications, see~\cite{Nak17,MW23} and references therein.}. 

To prove boundedness, we show that for each $\mathcal F$ on $X_\Gamma$, there is a ``minimal model'', from which $\mathcal F$ is pulled back, and then control the space of possible minimal models with fixed numerics. It is crucial not just that minimal models exist, but that they can be described in sufficiently explicit terms to be bounded. Theorem~\ref{thm: canonical-transversalization}, guarantees that -- after passing to the category of logarithmic spaces -- minimal models exist and the proof shows they are computed by $K$-tropicalizations. 

Fixing the $K$-tropicalization, and choosing arbitrarily a scheme theoretic proper birational model of it, the space of transverse quotient sheaves is easily seen to be parameterized by an open subscheme of an ordinary Quot scheme. 

\subsubsection{Tropicalizations of transverse sheaves} A $K$-tropicalization fixes the minimal models of expansions $X_\Gamma\to X$ and the numerical $K$-theory class on it, at least up to a finite choice, that a quotient sheaf must occupy. For boundedness, we show these $K$-tropicalizations come in finitely many combinatorial types. 

We provide some context for results of this kind. The simplest is for ordinary tropicalizations of curves in toric varieties -- balanced $1$-dimensional polyhedral complexes in a vector space. Already here the results are far from elementary. If the vector space has dimension $2$ finiteness is provided by the combinatorics of the secondary polytope of Gelfand, Kapranov, and Zelevinsky~\cite{GKZ,Mi03}. For curves in higher-dimensional vector spaces, finiteness can be reduced to this case by an elegant projection argument, due to Nishinou--Siebert~\cite{NishSieb06}. See also the work of Yu~\cite{Yu14}.  The balancing condition is the essential input for these finiteness statements, and plays a key role in boundedness of logarithmic stable map spaces~\cite{ACMW17,gross2012logarithmic}.

We deal with two new phenomena. The milder one is that we allow tropical objects of arbitrary dimension. Secondary polytope arguments can be adapted to this setting, are essentially about boundedness for the as-yet-undefined logarithmic Chow variety. But this only deals with the simplest part of tropical boundedness. The second, more serious, issue is that $K$-tropicalizations are more complicated combinatorial objects than ordinary tropicalizations. A $K$-tropicalization can have ``embedded components'' -- polyhedral complexes of lower dimension embedded inside faces of higher dimension. The naive balancing condition actually {\it fails} in these lower dimensional faces. When dealing with $1$-dimensional subschemes, the embedded components are just points, so the phenomenon does not cause serious issues~\cite{LiWu,MR20}. 

We find a replacement for the balancing condition for $K$-tropicalizations. The insight (and the source of the name) is that the $K$-theory of toric varieties and toric fiber bundles provide the enhancement of the ordinary balancing condition. We extract these conditions and use them to prove Theorem~\ref{thm: combinatorial-boundedness}.

\subsection{$K$-tropicalizations} We believe that $K$-tropicalizations are natural in tropical geometry. To give a sense of how they enhance usual tropicalizations, consider a subscheme $Z$ of a torus $\mathbb G_m^r$. Its traditional tropicalization is a subset of the cocharacter lattice of $\mathbb G_m^r$. Its maximal faces, with respect to any polyhedral complex structure, carry a multiplicity. From our perspective, it is more natural to view the set theoretic tropicalization as an equivalence class of polyhedral complex structures on that set, where the equivalence relation is generated by subdivisions. Since we can always take common refinements, this is actually no more information than the set, but is the perspective that generalizes well.

The $K$-tropicalization is not a set-theoretic object; it is an equivalence class of polyhedral structures on a set---indeed, on the same underlying set as the ordinary tropicalization. But unlike the ordinary tropicalization, not all polyhedral structures are allowed: the scheme structure of $Z$ selects distinguished ones coming from the algebraic transversality condition. These special polyhedral structures, up to equivalence under subdivision, define the \emph{support} of the $K$-tropicalization. 

Each cell of the $K$-tropicalization, regardless of dimension, carries an integer-valued decoration. The traditional decoration on maximal faces is always positive. The $K$-decorations are Euler characteristics and so can be negative, and bounding the negativity is an implicit difficulty throughout. This additional information is geometrically natural -- for instance, the $K$-tropicalization determines the Hilbert polynomial of the closure of $Z$ in any toric variety with respect to any polarization.

\subsection{Future directions} Boundedness completes the foundations of logarithmic Quot spaces -- with the results of this paper, we now know they are proper logarithmic spaces with tropicalization given by a moduli space of $K$-tropical objects.  But as noted, the paper and its precursors~\cite{Logquot,MR20,MR23} are part of a broader study of moduli of sheaves under snc degenerations -- a ``logarithmic moduli space of stable sheaves''. A theory along these lines would certainly have potential applications in enumerative geometry, and perhaps find use more broadly. It seems likely that every object in this proposed logarithmic sheaf theory appears in the logarithmic Quot space. The next stage in the development is a study of how these sheaves interact with Bridgeland or Gieseker stability conditions. The logarithmic Picard group is another piece of this story~\cite{MolchoWise}. 

There is an immediate potential application in enumerative geometry. The Quot schemes of surfaces carry virtual classes, and the associated enumerative invariants were studied by Oprea and Pandharipande~\cite{OpreaPandharipande}. They conjecture a structure for these invariants with striking similarities to the Donaldson--Thomas theory of threefolds. It is reasonable to hope that these conjectures can be generalized to pairs, and related to the original ones by degeneration formulas, analogous to~\cite{MR23}. 

Another direction, within tropical geometry, is a systematic study of $K$-tropicalizations and their relationship to state, Chow, and secondary polytopes~\cite{KSZ92,Sturmfels96}. Another interesting link is with tropical schemes. The $K$-tropicalization remembers the Hilbert polynomial, and another object that does this is the scheme theoretic tropicalization constructed in~\cite{GG16}. The latter seems to contain more information, but it seems reasonable to guess that the scheme theoretic tropicalization recovers the $K$-tropicalization. In the toric case, the $K$-tropicalization is also closely related to operational $K$-theory, $K$-theoretic Minkowski weights, the polytope algebra, and piecewise exponential functions, parallel to the Chow story~\cite{AP15,KP08,Mor93,Shah22}. We leave further exploration to the interested reader.

\subsection{Roadmap} The paper may be read largely independently of the foundational work~\cite{Logquot}, subject to a few blackboxes for which we give working expositions. We do direct the reader to the introduction of that paper for the motivation and summary of the theory as it stands prior to this one. Section~\ref{sec: prelim} contains a reminder of the basics of logarithmic Quot spaces. We also collect some facts about $K$-theory of coherent sheaves at the end of Section~\ref{sec: prelim}. In Section~\ref{sec: transverse-collapse} we prove Theorem~\ref{thm: canonical-transversalization}. In Section~\ref{sec: K-tropicalizations} we explain what $K$-tropicalizations are and use the $K$-theory of toric fiber bundles to derive a version of the balancing condition. We include a geometric treatment of traditional balancing, and develop tools to help analyze the $K$-balancing condition. In Section~\ref{sec: finiteness-K} we show that $K$-balancing implies the combinatorial boundedness in Theorem~\ref{thm: tropical-finiteness}. We put the pieces together in the final section and conclude. 

\subsection{Conventions and terminology} Schemes and stacks appearing here will be locally of finite type over an algebraically closed field. All logarithmic structures will be fine and saturated, and typically obtained by an explicit morphism to an Artin fan. We assume familiarity with the basics of logarithmic geometry, and refer the reader to the introductory sections of~\cite{MR20,MR23}, of~\cite{ModStckTropCurve}, or to~\cite{ACMUW} for background. 

\subsection{Acknowledgements} We have benefited from numerous conversations with friends and colleagues. We thank L. Battistella, D. Bejleri, E. Dodwell, J. Laga, D. Maclagan, S. Payne, T. Scholl, and J. Wise for soundboarding, interest, and encouragement. We also thank B. Siebert, M. Talpo, and R. Thomas for their interest, and for keeping us informed of their insights on the problem considered here. Parts of this document originate in the PhD thesis of the first author, and we are very grateful to M. Gross and M. Ulirsch, who examined the thesis, for detailed reading and comments that contributed to the work here. The second author extends a warm thanks to D. Maulik for the collaboration~\cite{MR20} and countless discussions on the topics presented here. 

\noindent 
DR was supported by EPSRC New Investigator Grant EP/V051830/1 and EPSRC Frontier Research Grant EP/Y037162/1.

\section{Preliminaries}\label{sec: prelim}

We begin with some reminders on algebraically transverse sheaves and the logarithmic Quot functor, as well as some reminders on $K$-theory. 

\subsection{Algebraically transverse sheaves on expansions}\label{sec: sheaves-on-expansions} Let $(X,D)$ be a pair consisting of a smooth projective variety $X$ and a simple normal crossings divisor $D$. Denote the divisor components by $D_1,\ldots, D_k$. There is an associated Artin fan $\mathsf A(X,D)$ -- a zero dimensional smooth Artin stack. It is an open substack of $[\mathbb A^k/\mathbb G_m^k]$. There is a canonical map
\[
X\to [\mathbb A^k/\mathbb G_m^k]
\]
given locally by choosing equations for the components of $D$. The stack $\mathsf A(X,D)$ is the smallest open substack through which this map factors. 

\begin{definition}
A coherent sheaf $\mathcal E$ on $X$ is called {\it algebraically transverse} if $\mathcal E$ is flat over $\mathsf A(X,D)$.
\end{definition}

One might also call such sheaves {\it logarithmically flat}, but we prefer the more geometric terminology, consistent with~\cite{MR20}.

\begin{remark}[Basics of algebraic transversality]
We make a few remarks to give the reader a feel for algebraic transversality. For details see~\cite[Section~3]{Logquot}. To keep the notation compact, we focus on subschemes, rather than general quotients, but parallel statements hold. 

Let $Z\hookrightarrow X$ be an algebraically transverse subscheme of $(X,D)$. If $X'\to X$ is a logarithmic blowup, then the strict transform $Z'$ of $Z$ is algebraically transverse and equal to the total transform\footnote{See~\cite[\href{https://stacks.math.columbia.edu/tag/080C}{Tag 080C}]{stacks-project} for the definition of a strict transform for a coherent sheaf, for use in that case.}. This follows from stability of flat maps under base change and recognizing the total transform as a pullback along $\mathsf A(X',D')\to\mathsf A(X,D)$. 

Another consequence is that $Z$ coincides with the closure\footnote{The closure of a coherent sheaf is not well-defined by the closure of a quotient by a fixed coherent sheaf is always well-defined, see Remark~\ref{rem: closures}.} of $Z|_{X\setminus D}$. More generally, if $W\subset X$ is a closed stratum and $W^\circ$ is the corresponding locally closed stratum, then transversality implies the equality of subschemes:
\[
Z\cap W = \overline{Z\cap W^\circ}.
\]
\end{remark}

Algebraically transverse quotients of algebraically transverse coherent sheaves form the interior of the logarithmic Quot space. In the boundary, objects are represented -- non-uniquely -- by sheaves on expanded degenerations, following~\cite{Logquot,MR20}.\footnote{The definition of an expansion is a little different from~\cite{MR20}. That paper is specialised to $0$ and $1$-dimensional subschemes, so the expansions are non-proper assumed to not have codimension $2$ strata. } We recall the ingredients of this.

Recall that $(X,D)$ determines a cone complex $\Sigma(X,D)$, given by a union of faces in $\mathbb R_{\geq 0}^k$. There is a continuous map from the topological space of $X$ to the partially ordered set\footnote{A partially ordered set is a topological space in the Alexandrov topology -- open sets are uppersets.} of coordinate faces of $\mathbb R_{\geq 0}^k$, recording for each point the collection of divisors that contain that point. The cone complex $\Sigma(X,D)$ is the smallest union of faces containing the image of $X$.

A subdivision $\Sigma'\to \Sigma(X,D)$ of cone complexes determines a birational modification
\[
X'\to X
\]
by standard toroidal geometry~\cite{KKMSD}. We use this is as follows. Consider a subdivision of $\Sigma(X,D) \times\mathbb R_{\geq 0}$ such that the preimage of $0$ is unchanged. A subdivision $\pi\colon \Sigma'\to\RR_{\geq 0}$ of this form is determined by the preimage $\pi^{-1}(1)$, which is a polyhedral complex $\Gamma$ that decomposes $\Sigma(X,D)$. It will be useful to flexibly pass between $\Sigma'$ and this slice. The slice is slightly more convenient when discussing examples and drawing pictures. 

A polyhedral decomposition $\Gamma$ as above determines a family 
$$\mathcal X_\Gamma\to X\times\mathbb A^1\to \mathbb A^1.$$ 
Note that this is a toroidal family in the sense of~\cite{KKMSD} and so inherits a logarithmic structure. 

\begin{definition}
Let $(X,D)$ be a simple normal crossings pair with cone complex $\Sigma$. A {\it rough expansion} of $X$ along $D$ is the fiber over $0$ of an expansion $\mathcal X_\Gamma\to X\times\mathbb A^1\to \mathbb A^1$, and is denoted $X_\Gamma$. It is an {\it expansion} if it is reduced. 

We equip $0$ with the pullback of the toric logarithmic structure on $\mathbb A^1$ making it $\Spec \mathbb C_{\mathbb N}$ and equip $X_\Gamma$ with the pullback logarithmic structure from the toroidal structure on $\mathcal X_\Gamma$. 
\end{definition}

The reducedness condition is equivalent to the integrality of the coordinates of the vertices of $\Gamma$. 

The polyhedral complex $\Gamma$ is determined by the logarithmic structure on an expansion -- we can perform the construction directly on the Artin fan $\mathsf A(X,D)$ and so obtain an expansion
\[
\mathsf A_\Gamma\to\mathsf A(X,D)\times[\mathbb A^1/\mathbb G_m].
\]
The pullback of $\mathcal B\mathbb G_m$ in $[\mathbb A^1/\mathbb G_m]$ is a stack of dimension $-1$; we denote it $\mathsf B_\Gamma$. There is a smooth map
\[
X_\Gamma\to \mathsf B_\Gamma
\]
fitting into a pullback diagram 
\[
\begin{tikzcd}
X_\Gamma\arrow{r}\arrow{d} & \mathcal X_\Gamma\arrow{d}\\
\mathsf B_\Gamma\arrow{r} & \mathsf A_\Gamma.
\end{tikzcd}
\]
Note that $\mathsf B_\Gamma$ is not quite an Artin fan, and fails the definition in two ways -- it is reducible, and its components are $\mathbb G_m$-gerbes over Artin fans. 

When we speak of an expansion, we will have a logarithmic structure on it, so the Artin fan, tropicalization, and the stack $\mathsf B_\Gamma$, can be reconstructed -- the expansion $X_\Gamma$ knows the map to $\mathsf B_\Gamma$. 

\begin{definition}
A coherent sheaf $\mathcal F$ on an expansion $X_\Gamma$ is {\it algebraically transverse} if $\mathcal F$ is flat over $\mathsf B_\Gamma$. 
\end{definition}

\subsubsection{Dilations, subdivisions, and pullbacks} Consider an tropical expansion
\[
\Sigma'\to \RR_{\geq 0}
\]
with slice $\Gamma$. Given a morphism
\[
\mathbb R_{\geq 0}\to \mathbb R_{\geq 0}
\]
sending $1$ to $m$ we can pull back $\Sigma'$ to obtain 
\[
\Sigma''\to \RR_{\geq 0}. 
\]
From the point of view of slices, this amounts to dilating the polyhedral complex $\Gamma$. 

The reason for doing this comes from semistable reduction. After dilation, there are more lattice points in the polyhedral complex and so there are more subdivisions that give rise to expansions, rather than only rough expansions. This is a key combinatorial procedure in~\cite{KKMSD}. The $1\mapsto m$ map described above is a ramified base change on $\mathbb A^1$ of order $m$. 

If $\Gamma'$ is a subdivision of a dilation of $\Gamma$, there is an associated morphism of expansions
\[
X_{\Gamma'}\to X_\Gamma
\]
and since flatness is preserved by arbitrary base change, an algebraically transverse sheaf $\mathcal F$ pulls back to an algebraically transverse sheaf $\mathcal F'$ on $X_\Gamma'$. 

We informally refer to maps $X_{\Gamma'}\to X_\Gamma$ determined by dilation--subdivisions by as {\it further subdivisions}. 

\subsection{The logarithmic Quot functor} One can perform the constructions of the previous sections in families by using a little logarithmic geometry. Equip $(X,D)$ with its divisorial logarithmic structure and let $S$ be a logarithmic scheme. A family of {\it expansions} of $X$ along $D$ over $S$ is a diagram
\[
\begin{tikzcd}
(X\times S)_\Gamma\arrow{dr} \arrow{rr}& & X\times S\arrow{dl}\\
&S&
\end{tikzcd}
\]
such that (i) $(X\times S)_\Gamma\to X\times S$ is a logarithmic modification, (ii) $(X\times S)_\Gamma\to S$ is flat with reduced fibers, and (iii) for any morphism $\Spec \mathbb C_{\mathbb N}\to S$ from the standard logarithmic point, the pullback of $(X\times S)_\Gamma$ is an expansion. 

A {\it family of algebraically transverse sheaves over $(X\times S)_\Gamma$} is a flat family of sheaves $\mathcal F$ over $S$ whose restriction to every fiber is algebraically transverse. 

Families of algebraically transverse sheaves, presented as quotients of a fixed sheaf, are the basic objects parameterized by the logarithmic Quot functor. The last ingredient is an equivalence relation. 

\begin{definition}\label{def: quot-equivalence}
Let $\mathcal E$ be an algebraically transverse sheaf on $X$. Consider two families of expansions
\[
p\colon (X\times S)_\Gamma\to S \ \ \textnormal{and} \ \ p'\colon (X\times S)_{\Gamma'}\to S.
\] 
Consider quotients $[p^\star \mathcal E\twoheadrightarrow \mathcal F]$ and $[{p'}^\star\mathcal E\twoheadrightarrow \mathcal F']$ with $\mathcal F$ and $\mathcal F'$ algebraically transverse. Such families of quotients are {\it equivalent}, if there is a common logarithmic modification $(X\times S)_{\Gamma''}$ of the both total spaces, such that the pullbacks of these two quotients coincide. 
\end{definition}

An equivalence class of algebraically transverse quotients of $p^\star\mathcal E$ is called a {\it logarithmic quotient of $\mathcal E$} or a {\it logarithmic surjection of coherent sheaves}. 

The logarithmic Quot functor is the functor on logarithmic schemes whose value on a logarithmic scheme $S$ is the collection of equivalence classes of families of families logarithmic quotients of $\mathcal E$. 

We denote the logarithmic Quot functor by $\mathsf{Quot}(X|D,\mathcal E)$ or simply $\mathsf{Quot}(X,\mathcal E)$ when $D$ is clear from context. When $\mathcal E$ is $\mathcal O_X$, we use the notation $\mathsf{Hilb}(X|D)$. 

\subsection{Piecewise linear spaces and the Gr\"obner stratification}\label{sec: pl-G-strat}

Working with the equivalence relation described in the previous section naturally suggests a combinatorial gadget -- a {\it piecewise linear space} -- replacing cone complexes, that represents the equivalence class. Objects of this kind appear in two basic ways, first at the level of the objects of the moduli problem and second for the moduli space itself. We only require the former, which is easier to treat in isolation. For the latter, we note that there is a canonical piecewise linear stack that admits a subdivision by a cone complex, and each subdivision gives rise to an algebraic stack structure on the logarithmic Quot space $\mathsf{Quot}(X|D,\mathcal E)$. We take this as a black box. 

The ``official'' definition of a piecewise linear space is given in~\cite{Logquot}. For us, the following will suffice. We can understand a cone complex $\Sigma$ as specifying a topological space  with a locally closed stratification, strata being interiors of cones. A subdivision 
\[
\Sigma'\to \Sigma
\]
is a refinement of this stratification. 

\begin{definition}
A {\it piecewise linear space} subdividing $\Sigma$ is a stratification $\mathscr T$ of $\Sigma$ such that the closure of a stratum is a linear space in a cone of $\Sigma$ with the property that $\mathscr T$ can be refined by a cone complex subdivision $\Sigma'\to\Sigma$. 
\end{definition}

Again, in practice, this is applied to $\Sigma\times\mathbb R_{\geq 0}$.

Let $p\colon X_\Gamma\to X$ be an expansion of $X$ along $D$ over $\Spec\mathbb C_{\mathbb N}$, where $\Gamma$ is a polyhedral subdivision of $\Sigma$. Let $[q\colon p^\star\mathcal E\twoheadrightarrow \mathcal F]$ on $X_\Gamma$ be an algebraically transverse quotient. Following ideas of Cartwright~\cite{Cartwright}, the first author defines a canonical piecewise linear space subdividing $\Sigma\times \RR_{\geq 0}$, written as: 
\[
\mathscr T(q)\to \Sigma \times \RR_{\geq 0}.
\]
The rigorous construction is best done using commutative algebra and Gr\"obner theory for modules of sections of $\mathcal F$. The basic idea is as follows. Each point $w$ of $\Sigma \times \RR_{\geq 0}$ determines  a torus torsor over a locally closed stratum of $(X,D)$, given by subdividing $\Sigma\times\mathbb R_{\geq 0}$ by adding the ray $(w,1)$, and then deleting all strata except the locally closed interior of the exceptional. Denote the latter by $E_w$. 

The {\it initial degeneration at $w$} is the restriction to $E_w$ of the pullback of $\mathcal F$. The initial degenerations for different points $w$ and $w'$ in the interior of the same cone of $\Sigma\times \{1\}$ gives rise to torus torsors $E_w$ and $E_{w'}$. These are isomorphic as torsors, though not as bundles. See for instance~\cite{PayFibers}. 

We maintain the notation above.

\begin{definition}
The {\it Gr\"obner stratification} or {\it tropical support} of the quotient $[q]$ is the piecewise linear space $\mathscr T(q)$ defined by the property that two points $w$ and $w'$ in the interior of the same cone lie in the same stratum if and only if their initial degenerations coincide. 
\end{definition}

The Gr\"obner stratification, as defined, is a piecewise linear stratification of $\Sigma\times\mathbb R_{\geq 0}$. As previously, we can take the slice over $1$ produce a piecewise affine strati of $\Sigma$. Examples of such objects include {\it polyhedral}, not necessarily conical, decompositions of $\Sigma$.

\begin{remark}[Constant degeneration]
The Gr\"obner stratification will also be used in the following simpler context -- called the ``constant coefficient case'' in tropical geometry. Let $X'\to X$ be a logarithmic blowup, with $X'$ simple normal crossings and let $q$ be an algebraically transverse quotient of the pullback of $\mathcal E$. The Gr\"obner stratification can be defined either by crossing with $\Spec \mathbb C_{\mathbb N}$, or by using weighted blowups in place of the $(w,1)$-blowups in the definition above. In either case, the result is a piecewise linear space that admits a refinement by a {\it cone complex} rather than just a polyhedral complex.  
\end{remark}

A feature of the construction is that 
\[
\mathsf{Cone}(\Gamma)\to \Sigma\times  \RR_{\geq 0}
\]
is a refinement of $\mathscr T(q)$ . As with subdivisions by polyhedral complexes, it is convenient when drawing pictures to slice the picture at height $1$ in the last factor, to obtain an affine linear stratification of $\Sigma$. 

\begin{remark}[Tropical supports in this text]
To make this paper self-contained, we briefly explain how tropical supports and piecewise linear spaces will arise in this paper. Given an algebraically transverse sheaf $\mathcal F$ on $X_\Gamma\to X$, we can ask whether there is a minimal subdivision from which $\mathcal F$ is pulled back. A natural guess, motivated by the geometry, is that the subdivision is given by the tropical support. Since the tropical support is not a cone complex, the best statement one can hope for is that it can be pulled back from any tropical expansion subdividing it. This is the first result that we prove, see Section~\ref{sec: transverse-collapse}.

Second, a key statement from~\cite{Logquot} is the existence of a (functorial) map from families of logarithmic quotients to families of tropical supports. The map is determined what it does to $\Spec \mathbb C_{\mathbb N}$-points of the functor, and its well-definedness is checked there.

The upshot is that the assignment of the Gr\"obner stratification to a logarithmic quotient defines a morphism from the logarithmic Quot space onto a piecewise linear stack of ``tropical supports''. We show that the set of tropical supports in the image of a component of the logarithmic Quot space is finite. We decorate the tropical support with additional data to form the {\it $K$-tropcialization} defined shortly. We show there are  finitely many combinatorial types of $K$-tropicalizations with given numerical data. 
\end{remark}

\subsection{A $K$-theory toolbox}\label{sec: toolbox} We require basic facts about the Grothendieck groups of coherent sheaves. These will be used to study $K$-balancing on the tropicalization coherent sheaf, see Section~\ref{sec:TropDecorations}. 

\subsubsection{Definitions} Let $X$ be a quasi-projective scheme. The Grothendieck group $K_\circ(X)$ of coherent sheaves on $X$ is generated by classes $[\mathcal F]$ for each coherent sheaf on $X$, with the relation $[\mathcal F] = [\mathcal F']+[\mathcal F'']$ for each exact sequence
\[
0\to \mathcal F'\to\mathcal F\to\mathcal F''\to 0.
\]
The functor $X\mapsto K_\circ(X)$ is covariant for proper maps, with pushforward given by $f_\star[\mathcal F] = \sum (-1)^i [R^if_\star\mathcal F]$. Given a sheaf $\mathcal F$ on a $X$ that is proper over $\Spec k$, then
\[
f_\star[\mathcal F] = \chi(X,\mathcal F)\in K_\circ(\Spec k) = \mathbb Z.
\]
If $f\colon Y\to X$ is a {\it flat} morphism, then pullback of coherent sheaves is exact, and so descends to:
\[
f^\star\colon K_\circ(X)\to K_\circ(Y). 
\]
If $X$ is projective and $H$ is a very ample line bundle, then the Hilbert polynomial of a coherent sheaf
\[
\mathcal F\mapsto P_X(\mathcal F,t)\in \mathbb Z[t], \ \ \ P_X(\mathcal F,n) = \chi(X,\mathcal F\otimes H^n)
\]
is additive in short exact sequences and so gives a homomorphism $K_\circ(X)\to \mathbb Z[t]$. 

When $X$ is smooth, the group $K_\circ(X)$ coincides with the Grothendieck group $K^\circ(X)$ of vector bundles. In this case, there is a natural ring structure on $K_\circ(X)$ given by derived tensor product. Equivalently, one can write every class in terms of alternating sums of vector bundles and take ordinary tensor product. 

We occasionally need these results for Artin stacks. The Artin stacks we need to work with are always global quotients of toric varieties by tori, so it suffices to note that all of this can be done equivariantly. 

\subsubsection{A few useful facts}\label{sec: basic-formulae} We will use a few formulae in $K$-theory. First, if $D\subset X$ is a Cartier divisor in $X$, the ideal sheaf exact sequences gives
\[
[\mathcal O_D] = 1-[\mathcal O_X(-D)] \ \textnormal{ in } K_\circ(X).
\]
We will later be in the situation where $D$ is reducible, and would like to relate $[\mathcal O_D]$ to classes of the irreducible components. This can be done by ``inclusion/exclusion''. 

\noindent
{\bf Inclusion/exclusion formulas.} Suppose $X_0$ is a simple normal crossings variety. For our purposes, all such varieties appear as unions of boundary divisors in snc pairs. Suppose that $X_0 = Y_1\cup Y_2$. By explicitly comparing structure sheaves, we obtain an exact sequence
\[
0\to \mathcal O_{X_0}\to\mathcal O_{Y_1}\oplus\mathcal O_{Y_2}\to \mathcal O_{Y_1\cap Y_2}\to 0
\]
and therefore the formula
\[
[\mathcal O_{X_0}] = [\mathcal O_{Y_1}]+[\mathcal O_{Y_2}]-[\mathcal O_{Y_1\cap Y_2}].
\]
If $X_0$ is simple normal crossings and has irreducible components $W_1,\ldots, W_m$ be the irreducible components, we can inductively obtain
\[
[\mathcal O_{X_0}] = \sum_{I\subset [m]} (-1)^{|I|} \cdot [\mathcal O_{W_I}],
\]
where $W_I$ is the intersection of $W_j$ for $j\in I$, and the sum ranges over all subsets. 

We use this in the following way. Consider the Artin fan $[\mathbb A^m/\mathbb G_m^m]$ and write $\mathsf B_0$ be the union of toric boundary divisors in it. The formula above holds expressing the structure sheaf of $\mathsf B$ in terms of its components, e.g. by observing that the formula holds in the equivariant theory of $\mathbb A^n$. 

We will also find ourselves in the following situation. We will have a space $\mathcal Y_0$ with a flat map
\[
\mathcal Y_0\to \mathsf B_0,
\]
and a coherent sheaf $\mathcal F_0$ on $\mathcal Y_0$ that is flat over $\mathsf B_0$. The space $\mathcal Y_0$  will have components $Y_i$. Since the sheaf is flat, we can pull back the formula for the structure sheaf of $\mathsf B_0$ and tensor wth $\mathcal F_0$. Denoting $\mathcal G_I$ for the restriction of $\mathcal F_0$ to $Y_I$, we have the following $K$-theory formula:
\[
[\mathcal F_0]= \sum_{I\subset [m]} (-1)^{|I|} \cdot [\mathcal G_I].
\]
When $Y_0$ is proper we can pass this to Euler characteristic. 

\noindent
{\bf Regular crossings divisors.} A useful generalization of normal crossings arises in our analysis, which Aluffi dubs ``regular crossings''.

\begin{definition}
Let $S$ be a scheme and $D_1,\ldots,D_m$ be Cartier divisors with union $D$. We say that the pair $(S,D)$ is {\it regular crossings}\footnote{The reader is warned of notational danger -- in the simple normal crossings case, one can recover the $D_i$ from $D$ by taking irreducible components, whereas that is not the case here. The decomposition of $D$ into $D_i$ is part of the data but suppressed from the notation.} if for each point $s\in S$, equations $f_1,\ldots,f_m$ cutting out the $D_i$ locally near $s$ form regular sequences in $\mathcal O_{S,s}$. 
\end{definition}

In other words, the divisors ``intersect completely'' at every point. Unlike in the simple normal crossings case, there are no smoothness restrictions. The regular crossings property is equivalent to the moduli map
\[
S\to [\mathbb A^m/\mathbb G_m^m]
\]
being flat~\cite[\href{https://stacks.math.columbia.edu/tag/07DY}{Tag 07DY}]{stacks-project}. If the map is smooth, then the pair is simple normal crossings. The key point for us -- an algebraically transverse subscheme of an snc pair is a regular crossings pair with respect to the pullback. 

The following property is useful.

\begin{lemma}
If $(S,D = \sum D_i)$ is a regular crossings pair then $(S,\sum a_iD_i)$ for any positive integers $a_i$ is also regular crossings.
\end{lemma}

\begin{proof}
If $f_1,\ldots,f_m$ form a regular sequence, then so do powers $f_i^{a_i}$ for any positive integers $a_i$. Alternatively, the map
\[
[\mathbb A^m/\mathbb G_m^m]\to [\mathbb A^m/\mathbb G_m^m]
\]
given by taking $a_i$-power in the $i^{th}$ coordinate is flat. Indeed, all fibers are dimension $0$, so this follows from miracle flatness. Compose with the map $S\to [\mathbb A^m/\mathbb G_m^m]$ that defines $D$ to conclude.
\end{proof}

The inclusion/exclusion formula for $K$-theory classes can be applied to any $X_0$ that is a union of Cartier divisors with regular crossings. 

\noindent
{\bf Multiplicities}. Several of our computations will take place in the $K$-theory of a smooth and projective toric variety (or in toric variety bundles over a base). The relations in this theory are provided by the fact that the line bundle on a toric variety $X$ associated to a character is trivial (or pulled back from the base, in the bundle case). In the toric case, for instance, the relation can also be expressed as an equality
\[
[\mathcal O_{ \sum a_i D_i}] =  [\mathcal O_{\sum b_j D_j}]
\]
of effective torus invariant Cartier divisors. Since the $D_i$ have simple normal crossings, the divisors $\sum a_i D_i$ and $\sum b_j D_j$ have regular crossings, and their $K$-theory classes can be understood via inclusion/exclusion as above. 

We would like to go further, to relate the structure sheaf of $a_i D_i$ to that of $D_i$ itself. 

The following discussion works on any scheme $X$ and Cartier divisor $D$. Let $m$ be a positive integer. Notice that $[\mathcal O_{mD}]$ is not, in general, equal to $m\cdot [\mathcal O_{D}]$. However, they are equal ``to first order''. 

Set $L = \mathcal O_X(-D)$. Using the ideal sheaf sequence for $\mathcal O_{mD}$, we have
\[
\mathcal O_{mD} = (1- L^m) = \mathcal O_D(1+L+\cdots+L^{m-1}). 
\]
Replace $L$ with $(1-\mathcal O_D)$, and induct on $m$ to get a formula of the form:
\[
\mathcal O_{mD}  = m\cdot \mathcal O_{D}+[\text{higher order terms in } \mathcal O_D].
\]
Higher order terms here are things of the form $[\mathcal O_D\otimes^{\bf L}\mathcal O_D]$, and further derived tensor powers. The precise formulae are left to the reader as they are not important for us, but for concreteness we have:
\[
\mathcal O_{2D} = 2[\mathcal O_D]-[\mathcal O_D\otimes^{\bf L}\mathcal O_D], \ \ \ \mathcal O_{3D} = 3[\mathcal O_D]-3[\mathcal O_D]^2+[\mathcal O_D]^3.
\]

The derived tensor powers should be viewed as {\it higher codimension} terms. They may be computed by using linear equivalences of divisors. If $\mathcal O_X(-D)$ is isomorphic to $\mathcal O_X(-D')$ such that $D$ and $D'$ meet algebraically transversely, i.e. all higher Tor's between $\mathcal O_D$ and $\mathcal O_{D'}$ vanish, then
\[
[\mathcal O_D]^2 = [\mathcal O_{D\cap D'}].
\]
A useful observation is that these higher order terms can be expressed as linear combinations of structure sheaves of known higher codimension strata. In the toric setting, we can represent $[\mathcal O_D]^k$ as a linear combination of positive  codimension strata of $D$, viewed as a toric variety in its own right. 

In the toric bundle setting, an analogous statement is true, where the strata are supported on intersections of pullbacks of hypersurfaces from the base and higher codimension toric strata. We will give explicit formulas in Section~\ref{sec: K-tropicalizations}.

\noindent
{\bf Bundles and blowups.} We close our toolbox by recording a couple of basic facts about bundles and blowups. We start with projective bundles.

\begin{lemma}
Let $X$ be a projective variety and let $p\colon \mathbb P_X\to X$ be a projective bundle. Suppose $\mathcal F$ is a sheaf on $X$. Then there is an equality of Euler characteristics:
\[
\chi(X,\mathcal F) = \chi(\mathbb P_X,p^\star\mathcal F). 
\]
\end{lemma}

\begin{proof}
We want to calculate the push forward of the class $[p^\star \mathcal F\otimes \mathcal O_{\mathbb P_X}]$ to a point. The map $p$ has connected fibers so $p_\star \mathcal O_{\mathbb P_X}$ is $\mathcal O_X$. By cohomology and base change, all higher $R^ip_\star\mathcal O_{\mathbb P_X}$ vanish, so the $K$-theory pushforward $p_\star [\mathcal O_{\mathbb P_X}]$ is $[\mathcal O_X]$. The result follows from the projection formula in $K$-theory. 
\end{proof}

We will also need control over blowups. 

\begin{lemma}\label{lem: blowup-formula}
Let $(X,D)$ be a simple normal crossings pair and let $\mathcal F$ be a sheaf on $X$ that is algebraically transverse to $D$. Let $p\colon X'\to X$ be the blowup of $X$ at a stratum of $D$. Then there is an equality of Euler characteristics:
\[
\chi(X,\mathcal F) = \chi(X',p^\star\mathcal F). 
\]
\end{lemma}

\begin{proof}
Let $W\subset X$ be the stratum. Consider the degeneration to the normal cone of $W$ in $X$, namely
\[
\mathsf{Bl}_{W\times 0} X\times\mathbb A^1\to X\times\mathbb A^1\to X.
\]
These maps are all pulled back from corresponding maps of relative Artin fans, so the pullback of $\mathcal F$ to the degeneration is a transverse sheaf in the total space. The restriction of $\mathcal F$ to any fiber has the same Euler characteristic. On the general fiber we recover the left hand side of the equation. For the right hand side, apply the inclusion/exclusion formula and then the previous lemma to conclude. 
\end{proof}

\section{Minimal models for transverse sheaves}\label{sec: transverse-collapse}

In this section we prove that canonical transversalization exists, as stated in Theorem~\ref{thm: canonical-transversalization}. 

Let $(X,D)$ be a pair consisting of a smooth projective variety and a simple normal crossings pair. As discussed before, we have a strict map
\[
X\to\mathsf A({X,D}).
\]
Fix a sheaf $\mathcal E$ on $X$ that is algebraically transverse to $D$, so $\mathcal E$ is flat over $\mathsf A(X,D)$. 

Let $\mathbb P_Z\to Z$ be a projective bundle. A sheaf $\mathcal G$ on $\mathbb P_Z$ is {\it vertically trivial} if it is the pullback of a sheaf on $Z$. The key result of this section is the following. 

In the following theorem, we understand the word {\it stratum} to refer to be a closed stratum with its reduced structure.

\begin{theorem}\label{thm: general-collapse}
Let 
\[
p\colon X' = \mathsf{Bl}_Z X \to X
\]
be a blowup of $X$ at a stratum $Z$ of $(X,D)$ with exceptional divisor $E$. Consider a quotient sheaf
\[
p^\star \mathcal E\to \mathcal F'
\]
and assume that $\mathcal F'$ is algebraically transverse on $(X',D')$ and that $\mathcal F'$ is vertically trivial along $E\to Z$. Then there is a sheaf $\mathcal F$ on $X$, algebraically transverse to $D$, which pulls back to $\mathcal F'|_{E}$, i.e. $\mathcal F'$ and $p^\star\mathcal F$ coincide as quotients of $p^\star \mathcal E$. 
\end{theorem}

By work Tevelev, given {\it any} coherent sheaf $\mathcal F$ on $X$, there exists a sequence of strata blowups $X''\to X$ such that the strict transform of $\mathcal F$ is algebraically transverse~\cite{Logquot,tevelev2005compactifications,TevNotes,TV21}. The process is constructive but non-canonical -- even when the starting sheaf is algebraically transverse the construction might perform further blowups. 

The result allows us to start with such a transversalization and blow down the loci that look like they can be blown down, to arrive at something minimal, since any cells that strictly refine the Gr\"obner stratification correspond to strata on which the quotient is vertically trivial. 

\noindent
{\it Deducing Theorem~\ref{thm: canonical-transversalization}}. We note Theorem~\ref{thm: general-collapse} implies Theorem~\ref{thm: canonical-transversalization} by the following elementary argument. Fix a quotient $[p^\star \mathcal E\twoheadrightarrow \mathcal F']$ on modification $X'\to X$. Let $\Delta\to \Sigma$ be any smooth fan that refines the Gr\"obner stratification of this quotient. We can restrict the quotient to the interior of $X_\Delta$ and take closure. Our goal is to show that the closure is algebraically transverse. Since sequences of stellar subdivisions are cofinal among subdivisions, we can perform a sequence of blowups of $X_\Delta$ along smooth centers until the result, say $X_\Delta'$ dominates $X'$. By algebraic transversality, the closure of the quotient on $X_\Delta'$ is the pullback of the quotient on $X'$. We use the collapsing result above to show that the quotient descends to an algebraically transverse one along $X_\Delta'\to X_\Delta$. By induction on the number of blowups, it suffices to treat the case of a single blowup. Since any subdivision of $X_\Delta$, by hypothesis, refines the Gr\"obner stratification, the restriction of the quotient to each exceptional is vertically trivial. The result follows.  \qed

\begin{remark}
    The smoothness hypothesis seems to be technical, but we have not managed to remove it. Since algebraically transverse sheaves pullback to algebraically transverse sheaves (meaning ones that are flat over the Artin fan) under arbitrary logarithmic blowups, the result that actually holds is the following. If $\Sigma'\to \Sigma$ is any, possibly non-smooth, subdivision that refines some smooth fan structure on the Gr\"obner stratification, the closure of the quotient in the corresponding $X'$ is algebraically transverse. 
\end{remark}

\subsection{Local and toric collapse} We state local, and then projective toric analogues of Theorem~\ref{thm: general-collapse}. The toric statement is proved first, using toric degenerations and Hilbert polynomial calculations; the latter needs projectivity. We deduce the affine local case, and finally the general case. 

We temporarily reset notation to:
\[
X = \mathbb A^r, \ \ X' = \mathsf{Bl}_Z X
\]
where $Z$ is a coordinate linear subspace and $p\colon X'\to X$ is the blowup. We take $D$ and $D'$ to be the complement of the tori in $X$ and $X'$.  Let $\mathcal E$ be a trivial module of finite rank. Consider a quotient
\[
p^\star\mathcal E\twoheadrightarrow\mathcal F'
\]
on $X'$ that is algebraically transverse. By restriction we obtain a quotient sheaf on $X'\setminus E$, which we identify with a quotient sheaf on $X\setminus Z$. By taking closure in $X$, see Remark~\ref{rem: closures}, we obtain:
\[
\mathcal E\twoheadrightarrow \mathcal F.
\]

Our goal is to prove the following: 

\begin{proposition}[Affine collapse]\label{prop: local-collapse}
The sheaf $\mathcal F$ is algebraically transverse to $D$ and its pullback coincides with $\mathcal F'$. 
\end{proposition}

\begin{remark}[Closures of quotients]\label{rem: closures}
Closures for quotients of coherent sheaves can be defined analogously to closures of a subscheme. We say a word about it, as it doesn't seem to make frequent appearances in the literature. Let $U$ be an open subscheme in $X$. Consider a surjection of sheaves on $U$ written $$q\colon \mathcal{O}_U^n \twoheadrightarrow \mathcal{F}_U.$$ Its closure in $X$ written $$\overline{q}\colon \mathcal{O}_X^n \twoheadrightarrow \mathcal{F}$$ is the unique surjection of sheaves with the following properties.

\begin{enumerate}[(i)]
    \item The surjection $\overline{q}$ restricts to $q_U$.
    \item The image $\mathcal{F}$ has no sections supported on the boundary of $X$.
\end{enumerate}

The closure exists and is terminal among surjections $q'\colon\mathcal{O}_X^n \rightarrow \mathcal{F}$ with the given restriction to $U$. It suffices to construct the closure for each affine open. On an affine open $V$, we specify an exact sequence $$ 0 \rightarrow K_V \rightarrow \mathbb{C}[V]^n\xrightarrow{\overline{q}} F_V\rightarrow 0.$$ The surjection $q$ is the data of a sequence of sheaves on $U \cap V$:
    $$ 0 \rightarrow K_{U \cap V} \rightarrow \mathbb{C}[U\cap V]^n\xrightarrow{q} F_{U,V}\rightarrow 0.$$
    The kernel $K_V$ is then specified by $$K_V = \{f \in \mathbb{C}[V]^n| f|_{U \cap V} \in K_{U \cap V}\}.$$ The surjection $\overline{q}$ is terminal because $K_V$ is maximal among submodules that pull back to $K_{U \cap V}$.
\end{remark}

We prove Proposition \ref{prop: local-collapse} via the following statement in the toric setting. 

\begin{proposition}[Toric collapse]\label{prop: toric-collapse}
Let $\mathcal Y$ be a smooth projective toric variety and $\mathcal Y'\to \mathcal Y$ be a blowup at a stratum $Z$ with exceptional divisor $E$. Let $q' \colon \mathcal O_{\mathcal Y'}^n \twoheadrightarrow\mathcal G'$ be a surjection of sheaves. Assume that $\mathcal G'$ is transverse to the toric boundary and $\mathcal G'|_E$ is vertically trivial along $E\to Z$. Then $q'$ is pulled back from a surjection $q\colon \mathcal{O}_X^n\twoheadrightarrow \mathcal G$ on $\mathcal Y$. 
\end{proposition}

The proof proceeds in several steps. Let $q \colon \mathcal{O}_X^n\rightarrow \mathcal G$ be the closure in $\cY$ of the restriction of $\cG'$ to $\mathcal Y'\setminus E$, namely $\cY\setminus Z$. The heart of the matter is showing that $\mathcal G$ is transverse. Once this is true, since both $q'$ and the pullback of $q$ to $\mathcal Y'$ are closures of their restrictions to the torus, so they coincide. 

We proceed by induction on the dimension of $\cY$, and so assume the proposition for all toric varieties of strictly smaller dimension. We need the following elementary claim.

\noindent
{\bf Claim.} Let $\sigma$ be the cone in the fan $\Sigma_{\mathcal Y}$ that is dual to the closed stratum $Z$. Let $\sigma_1,\ldots, \sigma_s$ be the maximal cones that contain $\sigma$ as a face. In other words, they are the maximal cones in the star of $\sigma$. Let $U_i\subset\mathcal Y$ be the torus invariant open dual to $\sigma_i$. In order to show that $\mathcal G$ is transverse, it suffices to show that $\mathcal G|_{U_i}$ is transverse to the boundary of $U_i$ for each $i$. 

\noindent
{\it Proof.} Flatness of a sheaf can be checked affine locally. \qed

Fix a cone $\sigma_i$ and the corresponding open $U_i$. Consider the toric rational map
\[
\mathcal Y\dashrightarrow \mathbb P^1
\]
defined on fans by mapping each ray of $\sigma_i$ to the positive ray in $\mathbb P^1$ with slope $1$. Since $\sigma_i$ is a smooth cone, this prescription defines a linear function on the cocharacter lattice, from which we obtain the rational map above. The open set $U_i$ is, by construction, contained in the domain. Let $U_i'$ be its preimage in $\cY'$. Note that $U_i'$ is open but not affine. 

We require a second claim.

\noindent
{\bf Claim.} The rational map $\mathcal Y\dashrightarrow \mathbb P^1$ can be resolved to $\mathcal Y_\sim\to\mathbb P^1$, with $\mathcal Y_\sim$ smooth, by toric blowups with center in the complement of $U_i$. 

\noindent
{\it Proof.} The indeterminacy locus is torus invariant because the rational map is induced by a homomorphism of tori. Blowup the indeterminacy locus and apply toric resolution of singularities. The former clearly doesn't change $U_i$, while the latter can be chosen not to change $U_i$ by functoriality of resolution of singularities. \qed

Let $\cY'_\sim$ be the blowup of $\cY_\sim$ at the strict transform of the center $Z$. Naturally, on the preimage of $U_i$ this this is isomorphic to $\cY'$. In order to avoid getting burdened with notation, we {\it replace} $\cY$ and $\cY'$ with $\cY_\sim$ and $\cY'_\sim$ respectively. We correspondingly replace $q\colon \mathcal{O}_X^n\rightarrow \cG$ and $q\colon \mathcal{O}_X^n \rightarrow \cG'$ with their strict transforms. In the latter case, it is also the total transform since that sheaf is transverse. We have changed the problem slightly by doing this, as we have replaced the original $\cG$ with its strict transform. But since both sheaves restrict to the given ones on $U_i$ and $U_i'$, for the purposes of the claim, the situation is unchanged. In particular, if we show that for each $U_i$ this new replacement $\cG$ is transverse, then we deduce the statement for the original $\cG$. With this understanding, we will prove algebraic transversality for a particular $U_i$, and therefore make the replacements. 

We have degenerations:
\[
\begin{tikzcd}
U_i'\subset \mathcal Y'\arrow{dr}\arrow{rr} && U_i\subset \mathcal Y\arrow{dl}\\
&\mathbb P^1.&
\end{tikzcd}
\]
The image of $Z$ is a torus invariant point of $\mathbb P^1$, so we choose coordinates to make it $0$. The affine open set $U_i$ lies over the invariant affine patch of $\mathbb P^1$ containing $0$. 

Since $U_i$ is a smooth affine toric variety corresponding to a full dimensional cone, it has no torus factors and so is isomorphic to $\mathbb A^r$, where $r$ is the dimension of $\cY$. The local structure of the degeneration is as follows. Let $x_1,\ldots,x_k,y_{k+1},\ldots,y_r$ be the coordinates on $U_i$ and let $x_1,\ldots,x_k$ be the coordinates whose vanishing locus is $Z$. Consider the map
\[
X = \mathbb A^r\to\mathbb A^1.
\]
given by the product of all coordinates. The map exhibits $\A^r$ as the total space of a simple normal crossings degeneration over $\A^1$. The special fiber of the degeneration is the toric boundary
\[
\left\{(x_1,\ldots,x_k,y_{k+1},\ldots,y_r)\colon \textnormal{at least one of the $x_i$'s or $y_i$'s is $0$} \right\}.
\]
The general fiber is isomorphic to $\mathbb G_m^{r-1}$. The special fiber consists of $r$ irreducible components, each isomorphic to $\mathbb A^{r-1}$. By composition, the blowup is also presented as a degeneration, fitting into:
\[
\begin{tikzcd}
U_i'\arrow{dr}\arrow{rr} && U_i\arrow{dl}\\
&\A^1.&
\end{tikzcd}
\]
The general fibers, and more precisely the preimage of $\mathbb G_m\subset\mathbb A^1$, are the same. The special fiber of $U_i'\to \A^1$ is non-reduced along the exceptional divisor $E$ and this component has multiplicity equal to $k$. The special fiber is given by
\[
\mathbb V(\Pi  (\widetilde x_i\cdot s)\cdot \Pi \ y_j)\subset X'_0,
\]
where $\widetilde x_i$ is the section of the line bundle on $U_i'$ that cuts out the strict transform of the $i^{th}$-coordinate hyperplane.

We turn our attention to the sheaves $\mathcal G$ and $\mathcal G'$. We would like to view $\mathcal G'$ and $\mathcal G$ as flat families of sheaves on these two degenerations. The sheaves are flat over $[\mathbb P^1/\mathbb G_m]$, though they need not be flat over $\mathbb P^1$ itself\footnote{For example, if the sheaves had an isolated closed point in the interior of the torus, this would not affect flatness over $[\mathbb P^1/\mathbb G_m]$ but over $\mathbb P^1$, flatness would fail over the image of this point.}. We make do with the following. 

\begin{lemma}
There is an open subscheme $U\subset \mathbb P^1$ containing $0$, such that the pullbacks 
\[
\mathcal G'|_U \ \ \textnormal{and} \ \ \mathcal G|_U
\]
are flat families of sheaves over $U$. 
\end{lemma}

\begin{proof}
Both $\mathcal G$ and $\mathcal G'$ are certainly flat over some open subset of $\mathbb P^1$. The only nontrivial aspect to is to check that this open can be chosen to contain $0$. For this, it is enough to show that there are no sections with support that lie in the preimage of $0$. The preimage of $0$ lies in the boundary of $\cY$ resp. $\cY'$, and since $\mathcal G$ resp. $\mathcal G'$ coincide with the closures of their interiors, there are no sections supported over the boundary. 
\end{proof}

The quotient sheaf $\mathcal G$ has been defined by closure of a quotient on the complement of $Z$. Therefore, a priori, we have no control over its behaviour near the center $Z$ and this is the essential aspect of the problem. However, by the above lemma, the restriction of $\mathcal G$ to the special fiber $\mathcal Y_0$ is the flat limit of the restriction of $\mathcal G$ to a generic fiber of $\mathcal Y\to\mathbb P^1$. This gives us some control over the closure -- its Hilbert polynomial is known in terms of data associated to the sheaf $\mathcal G'$, and this is how we will control it. 

We organize the calculation by the following:

\noindent
{\bf Construction.} We build a sheaf $\mathcal H_0$ on $\mathcal Y_0$ as follows. It will serve as a candidate to describe the special fiber $\mathcal G|_{\mathcal Y_0}$. Take $q\colon \mathcal{O}_X^n\twoheadrightarrow \cG'$ and restrict it to the locally closed subscheme $\mathcal Y_0'\setminus E$. The blowup map identifies this locally closed subscheme with $\mathcal Y_0\setminus Z$. Define $q_H\colon \mathcal{O}_\mathcal{Y}^n\rightarrow \mathcal H_0$ to be the closure of this restriction in $\mathcal Y_0$. 

Let $\mathsf A_{\cY'}$ denote the Artin fan $[\cY'/T]$ and let $\mathsf B_{\cY'_0}$ be the closed substack corresponding to $\cY'_0$. We use analogous notation for $\cY$. Note that $\cY'_0$ is the vanishing of a monomial, so is torus invariant, and therefore uniquely defines a closed substack of $\mathsf A_{\cY'}$. 

\noindent
{\bf Claim.} The sheaf $\cG'_0 = \cG'\otimes \mathcal O_{\cY_0}$ is flat over $\mathsf B_{\cY'_0}$. 

\noindent
{\it Proof.} The sheaf $\cG'$ is flat over $\mathsf A_{\cY'}$. We would like to base change to $\mathcal Y_0$, namely, to pullback along the inlcusion of $\mathsf B_{\mathcal Y_0}$. Flatness is preserved by arbitrary base change. \qed

This is the form of algebraic transversality we encounter for sheaves on expansions, outlined in the background section. We show that the candidate sheaf $\mathcal H_0$ also has this form of transversality. 

\begin{proposition}\label{prop: candidate-is-flat}
The sheaf $\mathcal H_0$ is flat over $\mathsf B_{\cY_0}$. 
\end{proposition}

Before we prove this statement, we require the following. 

\begin{lemma}\label{lem: flat-over-broken}
Let $\cY_0\to \mathsf B_{\cY_0}$ be as above, and let $Y_1,\ldots, Y_e$ be the irreducible components of $\cY_0$. Let $\cY_0^\circ$ be the union of the dense tori of the $Y_i$'s. A quotient sheaf $[\mathcal O_{\cY_0}^r\twoheadrightarrow \mathcal S_0]$ is flat over $\mathsf B_{\cY_0}$ if and only if the following two conditions hold for each $Y_i$: (i) the restriction of $\mathcal S_0$ to $Y_i$ agrees with the closure of $\mathcal S_0|_{Y_i^\circ}$, and (ii) this closure is algebraically transverse to the boundary of $Y_i$. 
\end{lemma}

\begin{proof}
Flatness descends through normalization by~\cite[\href{https://stacks.math.columbia.edu/tag/0533}{Tag 0533}]{stacks-project} so it suffices to check flatness over the irreducible components of $\mathsf B_{\cY_0}$. Each of these components is a $\mathcal B\mathbb G_m$-bundle over the Artin fan of the appropriate $Y_i$. Since our bundle is a smooth morphism, this is equivalent to flatness over the Artin fan corresponding to each component, which is component-wise algebraic transversality. The result is immediate. 
\end{proof}

\noindent
{\it Proof of Proposition~\ref{prop: candidate-is-flat}}. We check conditions (i) and (ii) of Lemma~\ref{lem: flat-over-broken} directly, in reverse order. That is, first we check that for each component $Y_i$ of $\mathcal Y_0$, we check that the closure of the restriction of $\mathcal G'$ to the dense torus of $Y_i$ is transverse. Then we check that each closure induces the same sheaf on the strata of $\mathcal Y_0$ where the different $Y_i$ meet. 

Fix a component $Y_i$. It has a strict transform $Y_i'$ in $\cY_0'$, and the morphism
\[
Y_i'\to Y_i
\]
is a blowup of $Y_i$ at its intersection with $Z$, which it may contain. Call the exceptional divisor of this blowup $E_i$. The exceptional divisor $E$ of the larger blowup $\cY'\to\cY$ meets $Y_i'$ exactly in $E_i$. 

The sheaf $\mathcal H_0$ is the closure of the restriction of $\mathcal G'$ to a locally closed subscheme -- the complement of $E$ -- in $\mathcal Y_0$. The restriction of $\mathcal H_0$ to the interior torus of $Y_i$ coincides with the restriction of $\mathcal G'$ to this locus. 

We have noted that the sheaf $\mathcal G'|_{Y_i'}$ is transverse as a sheaf on $Y'_i$. The further restriction $\mathcal G'|_{E_i}$ is vertically trivial. To see this, note that we can obtain $\mathcal G'_{E_i}$ can be obtained by first taking $\mathcal G'|_{E}$ and restricting to $E_i$. Since $\mathcal G'_{E}$ is already vertically trivial, the further restriction $\mathcal G'_{E_i}$ is also vertically trivial. Since $Y_i'$ has lower dimension, by the inductive hypothesis, we can use the statement of the theorem here to conclude that  there is a sheaf of $Y_i$ that is tranverse and pulls back to $\mathcal G'|_{Y_i}$. This verifies (ii). 

We now check the condition (i), namely that, for each $i$, if we restrict $\mathcal H_0$ to $Y_i$, we recover the closure of $\mathcal H_0|_{Y_i^\circ}$. Denote this closure by $\mathcal H_i$. Recall that the closure is defined as the quotient whose kernel consists of all sections over $Y_i$ that lie in the kernel  of $[\mathcal O_{\cY_0}^n\twoheadrightarrow \mathcal H_0]$ after restriction to $Y_i^\circ$. For this, we require a description of the sections of the sheaf $\mathcal O_{\cY_0}^n$, and so of $\mathcal O_{\cY_0}$. 

We have the normalization map
\[
\bigsqcup_i Y_i\to \cY_0.
\]
The structure sheaf of $\cY_0$ consists, locally, of tuples of functions $f_i$ on $Y_i$ that glue {\it topologically} -- if $Y_I$ denotes the intersection of $Y_i$ for $i\in I$, then all the $f_i$ for $i\in I$ must give the same function upon restriction to $Y_I$. If this holds for all $I$ then such a tuple of functions is an element of $\mathcal O_{\mathcal Y_0}$. This follows from recognizing the scheme $\cY_0$ locally as a union of coordinate hyperplanes in affine space. 

We will use this to verify condition (i), but before we can do that, we need another fact: 

\noindent
{\bf Claim.} For each subset $I\subset [r]$, we have a stratum $Y_I$ of $\cY_0$. We can then consider the restrictions $\mathcal H_i|_{Y_I}$ to the {\it closed strata} for each $i$. These restrictions for different $i$ coincide. 

\noindent
{\it Proof.} By the induction, each $\mathcal H_i$ is algebraically transverse as a sheaf on $Y_i$, so checking the claimed condition on closed strata is equivalent to checking the same restriction condition on {\it locally closed} strata $Y_I^\circ$. To be precise, the stratum $Y_I^\circ$ is the locus in $Y_I$ consisting of points that are contained in $Y_i$ for all $i\in I$, and are {\it not contained} in any $Y_j$ for $j\notin I$. The equivalence between the conditions on closed strata and locally closed strata follows from restriction properties of surjections and algebraic transversality. 

Consider the blowup $\mathcal Y'\to \mathcal Y$. For those strata $Y_I$ whose preimage in $\cY'$ is a stratum that maps isomorphically back onto it, the agreement condition we are after follows from the transversality of the sheaf $\mathcal G'$. The remaining strata $Y_I^\circ$ are contained in $Z$. For each $Y_i$ that contains $Z$, there is a toric stratum of $Y_i'$ that maps isomorphically onto $Y_I^\circ$. This locally closed stratum is necessarily also a toric stratum of $E$. By transversality of $\mathcal H_i$ on $Y_i$, this restriction to $Y_I^\circ$ can be obtained by restricting the sheaf $\mathcal G'$ to $Y_i$. In turn, this time by transversality of $\mathcal G'$, it can be obtained by restricting $\mathcal G'|_{E}$ and pushing forward along the isomorphism to $Y_I^\circ$. In other words, the restrictions of the different $\mathcal H_i$ to $\mathcal Y_I^\circ$ can be identified with the restriction of $\mathcal G'|_{E}$ to corresponding locally closed torus orbits of $E$ that lie over the same torus orbit in $Z$. We know by hypothesis that $\mathcal G'|_{E}$ is vertically trivial, and so is pulled back along $E\to Z$. The sheaf that pulls back to it is necessarily transverse. Indeed, this sheaf must be equal to the restriction of $\mathcal G'|_{E}$ to a coordinate section of the projective bundle $E\to Z$, and the latter is transverse by transversality of $\mathcal G'$. We therefore identify the different restrictions of the different $\mathcal H_i$ to $Y_I^\circ$ with the restriction of a transverse sheaf on $Z$ to some fixed single torus orbit, and the agreement follows. \qed

We finally verify condition (i) for each $i$. On $Y_i$ we have a surjection $[\mathcal O_{Y_i}^n\twoheadrightarrow \mathcal H_i]$ with kernel $\mathcal K_i$. In order to check the closure condition, we need to verify that for each $f_i$ in $\mathcal K_i$, there exists an element $f$ of $\mathcal O_{\cY_0}^n$ that restricts to $f_i$ on $Y_i$. We proceed component by component. Fix an $f_1\in \mathcal K_1$ and consider its restriction to $Y_{12}$. By the agreement condition above, the restriction is contained in the kernel $\mathcal K_{12}$ of the surjection $[\mathcal O_{Y_{12}}^n\twoheadrightarrow \mathcal H_2|_{Y_{12}}]$. By transversality on $Y_2$, we can find $f_2\in\mathcal K_2$ that agrees with $f_2$ upon restriction to $Y_{12}$. By the description of $\mathcal O_{\mathcal Y_0}$ above, this means we can find $f_{12}$ on  $\mathcal O_{Y_1\cup Y_2}^n$ that extends $f_1$ and lies in the kernel of the restriction of the quotient to $Y_1\cup Y_2$. Similarly, we can find $f_3\in\mathcal K_3$ that agrees with $f_{12}$ on $(Y_1\cup Y_2)\cap Y_3$, and so on and so forth. This implies condition (i). The proof of the proposition is complete. 
\qed

We have ended up with two sheaves on $\cY_0$. The first is $\mathcal G_0$ which is the flat limit of a general fiber of the family $\mathcal G|_{U\setminus\{0\}}$. The second is constructed ``by hand'' namely $\mathcal H_0$. The latter has transversality properties, while the former fits into a flat family with the general fiber $\mathcal G|_{U\setminus\{0\}}$. We show they coincide. 

\begin{proposition}\label{prop: candidate-is-special-fiber}
The sheaves $\mathcal G_0$ and $\mathcal H_0$ coincide. 
\end{proposition}

\begin{proof}
We first show that the Hilbert polynomials of $\mathcal G_0$ and $\mathcal H_0$, with respect to a projective embedding of $\cY$, coincide. By \cite[Lemma 1.2.1]{GeometryModuliSheaves}, it is enough to show the Euler characteristics of the restrictions of $\cG_0$ and of $\mathcal H_0$ to general linear sections of each dimension are the same.

We begin by comparing the Euler characteristics. Let $\mathcal G_\eta$ be a general fiber of $\mathcal G$, or equivalently of $\mathcal G'$, over $\mathbb P^1$. Since Euler characteristic is constant in flat families we learn:
\[
\chi(\cY'_0,\mathcal G'_0) = \chi(\cY_\eta,\cG_\eta) = \chi(\cY_0,\mathcal G_0).
\]
We must therefore prove equality of Euler characteristics for two sheaves which we know satisfy a transversality property: \begin{equation}\label{eqn:EqualEulerChar} \chi(\cY'_0,\mathcal G'_0) = \chi(\cY_0,\mathcal H_0).\end{equation} The transversality property will allow us to compute both Euler characteristics by the inclusion/exclusion formulae of Section~\ref{sec: toolbox}. 

Indeed, consider the $K$-theory of $\mathcal Y$ and in it, the class of the structure sheaf of the special fiber $\mathcal Y_0$. Since $\mathcal Y$ is smooth the toric divisors form regular sequences. The special fiber is cut out by a monomial function -- a product of powers of sections cutting out toric divisors. Since powers of regular sequences are regular sequences, we obtain a formula for the structure sheaf of $\mathsf B_{\cY_0}$ in $K$-theory -- it is the sum of the irreducible components, minus the double intersections, plus the triple intersections, and so on. Since $\mathcal H_0$ is flat over $\mathsf B_{\cY_0}$ we obtain a similar formula for $\mathcal H_0$ by pullback -- sum the Euler characteristics of the restrictions of $\mathcal H_0$ to irreducible components, subtract the restrictions to codimension $1$ strata, and so on and so forth. The analogous formula holds for the Euler characteristic of $\mathcal G'_0$ on $\cY_0'$ in terms of its restrictions to components and strata of $\cY_0'$. 

To establish equation (\ref{eqn:EqualEulerChar}) we compare the above formulae term-by-term. Note that $\cY'_0$ is typically non-reduced along the exceptional $E$, even when $\cY_0$ is reduced, so some care is required in working with it. Contributions to 
$\chi(\cY'_0,\mathcal G'_0)$ are indexed by the cones of $\Sigma_{\mathcal Y'}$ that map to the positive ray of $\RR$, and similarly contributions to $\chi(\cY_0,\mathcal H_0)$ are indexed by cones of $\Sigma_{\cY}$. 

Recall $\sigma$ is the stratum dual to $Z$ and let $\rho_E$ be the ray associated to the exceptional $E$. There is a bijection $B$ from cones of $\Sigma_{\mathcal Y'}$ which do not contain $\rho_E$ as a ray to cones in $\Sigma_{\mathcal Y}\backslash \{\sigma_Z\}$. Note that if $\mathcal F$ is a sheaf on $Y$ transverse to a normal crossings divisor, the pullback to a blowup $Y'\to Y$ at a stratum has the same Euler characteristic; see Section~\ref{sec: toolbox}. Thus $B$ preserves Euler characteristic associated to strata dual to all such cones, and contributions cancel termwise. 

We are left to show that $\chi (Z,\mathcal H_0|_Z)$ coincides with the contribution from exceptional cones. We now introduce notation so that we might state this equality explicitly. Write $\rho_1,\ldots, \rho_\ell$ for rays of $\sigma$. Note that each ray maps onto the positive ray in $\mathbb R$, induced by the degeneration, with some slope. These slopes determine the non-reduced structure along the strata of $\cY_0'$. By construction, each of the $\rho_i$ map with slope $1$, but $\rho_E$ maps with slope $\ell$. 

With this setup, we can state the identity required to equate the two Euler characteristics. Let $\chi_E$ be the Euler characteristic $\mathcal G'_0$ restricted to the divisor $\ell E$ and for a non-empty subset $J\subset \{1,\ldots,\ell\}$, let $\chi_{E,J}$ be the Euler characteristic of $\mathcal G'_0$ restricted to the stratum given by $\ell E\cap\bigcap j_{\in J} Y'_j$. Then we need to show that
\begin{equation}\label{eqn: exceptional equality}
\chi (Z,\mathcal H_0|_Z) = \chi_E +\sum_{J\subset [\ell]} (-1)^{|J|} \chi_{E,J}.
\end{equation}
The equality can be shown by direct calculation, but we avoid the resulting tedious analysis of binomial coefficents (an offence for which the non-reduced structure along $E$ should be blamed) and use a geometric argument. The geometric argument is recorded in Remark~\ref{rem: absurdity argument}.

Assuming Equation \ref{eqn: exceptional equality} we complete the proof. Choose a very ample line bundle on $\mathcal Y$. By a Bertini argument, the restriction of $\mathcal G$, $\mathcal G'$, and $\mathcal H_0$ to finitely many generic elements of any ample linear system remain flat over $B_{\mathcal{Y}_0}$ and $B_{\mathcal{Y}'_0}$ as appropriate. The above argument thus implies that the Hilbert polynomials of $\mathcal G_0$ and $\mathcal{H}_0$ coincide. Furthermore, since $\mathcal G'_0$ is a quotient that agrees with $\mathcal H'_0$ on the union of the interiors $\bigcup_i Y_i^\circ$, and the latter has been defined as a closure, the sheaf $\mathcal H_0$ is a quotient of $\mathcal G_0$. To complete the proof recall that any surjection between sheaves of the same Hilbert polynomial is an isomorphism. 
\end{proof}

We record the required identity from the proof. 

\begin{remark}\label{rem: absurdity argument} The basic idea is to build an auxiliary degeneration that contains the stratum $Z$ as a stratum in its special fiber, and furthermore where the normal bundle of $Z$ is the same as what occurs above. We then build a degeneration of a transverse sheaf that recreates $\mathcal H_0|_{Z}$ in the special fiber. Effectively, this exhbits the local geometry of $Z$ in the picture above as the special fiber of degeneration/smoothing. The advantage is that we have a general fiber that is unchanged by blowup of $Z$, with the same Euler characteristic, and since the Euler characteristic can be computed from the general fiber, we can compute the contribution of the offending non-reduced component indirectly. We now carry this out.

The normal bundle of the stratum $Z\subset \mathcal Y$ admits a natural splitting $L_1,\ldots,L_\ell$, let $L_i^\star$ denote the associated torus bundle, and let $\mathcal T = \oplus_i L_i^\star$. We produce an equivariant degeneration/compactification of $\mathcal T$ over $\mathbb A^1$. Such a degeneration corresponds to a polyhedral subdivision of the cocharacter space $N_\RR$ of the fiber torus of $\mathcal T$. The subdivision is constructed by starting a full-dimensional unimodular simplex in $N_\RR$, and at each vertex, adding a ray to make the star of the vertex isomorphic to the fan of projective space $\mathbb P^{\ell-1}$. The recession fan is the fan of $\mathbb P^{\ell-1}$. See Figure~\ref{fig: degeneration-fan} for an illustration when $\ell = 3$.
\begin{figure}[h!]

\tikzset{every picture/.style={line width=0.75pt}} 

\begin{tikzpicture}[x=0.75pt,y=0.75pt,yscale=-1,xscale=1]

\draw    (103,87) -- (103,165) ;
\draw [shift={(103,165)}, rotate = 90] [color={rgb, 255:red, 0; green, 0; blue, 0 }  ][fill={rgb, 255:red, 0; green, 0; blue, 0 }  ][line width=0.75]      (0, 0) circle [x radius= 3.35, y radius= 3.35]   ;
\draw [shift={(103,87)}, rotate = 90] [color={rgb, 255:red, 0; green, 0; blue, 0 }  ][fill={rgb, 255:red, 0; green, 0; blue, 0 }  ][line width=0.75]      (0, 0) circle [x radius= 3.35, y radius= 3.35]   ;
\draw    (103,165) -- (173,165) ;
\draw [shift={(173,165)}, rotate = 0] [color={rgb, 255:red, 0; green, 0; blue, 0 }  ][fill={rgb, 255:red, 0; green, 0; blue, 0 }  ][line width=0.75]      (0, 0) circle [x radius= 3.35, y radius= 3.35]   ;
\draw    (103,87) -- (173,165) ;
\draw    (64.99,207.76) -- (103,165) ;
\draw [shift={(63,210)}, rotate = 311.63] [fill={rgb, 255:red, 0; green, 0; blue, 0 }  ][line width=0.08]  [draw opacity=0] (8.93,-4.29) -- (0,0) -- (8.93,4.29) -- cycle    ;
\draw    (222.34,190.62) -- (173,165) ;
\draw [shift={(225,192)}, rotate = 207.44] [fill={rgb, 255:red, 0; green, 0; blue, 0 }  ][line width=0.08]  [draw opacity=0] (8.93,-4.29) -- (0,0) -- (8.93,4.29) -- cycle    ;
\draw    (84.11,39.79) -- (103,87) ;
\draw [shift={(83,37)}, rotate = 68.2] [fill={rgb, 255:red, 0; green, 0; blue, 0 }  ][line width=0.08]  [draw opacity=0] (8.93,-4.29) -- (0,0) -- (8.93,4.29) -- cycle    ;
\draw    (390.62,66.93) -- (404,130) ;
\draw [shift={(404,130)}, rotate = 78.02] [color={rgb, 255:red, 0; green, 0; blue, 0 }  ][fill={rgb, 255:red, 0; green, 0; blue, 0 }  ][line width=0.75]      (0, 0) circle [x radius= 3.35, y radius= 3.35]   ;
\draw [shift={(390,64)}, rotate = 78.02] [fill={rgb, 255:red, 0; green, 0; blue, 0 }  ][line width=0.08]  [draw opacity=0] (8.93,-4.29) -- (0,0) -- (8.93,4.29) -- cycle    ;
\draw    (457.22,151.86) -- (404,130) ;
\draw [shift={(460,153)}, rotate = 202.33] [fill={rgb, 255:red, 0; green, 0; blue, 0 }  ][line width=0.08]  [draw opacity=0] (8.93,-4.29) -- (0,0) -- (8.93,4.29) -- cycle    ;
\draw    (358.99,180.76) -- (404,130) ;
\draw [shift={(357,183)}, rotate = 311.57] [fill={rgb, 255:red, 0; green, 0; blue, 0 }  ][line width=0.08]  [draw opacity=0] (8.93,-4.29) -- (0,0) -- (8.93,4.29) -- cycle    ;
\draw    (210,120) .. controls (249.2,90.6) and (269.19,147.64) .. (307.62,121.69) ;
\draw [shift={(310,120)}, rotate = 143.13] [fill={rgb, 255:red, 0; green, 0; blue, 0 }  ][line width=0.08]  [draw opacity=0] (10.72,-5.15) -- (0,0) -- (10.72,5.15) -- (7.12,0) -- cycle    ;

\draw (227,142) node [anchor=north west][inner sep=0.75pt]   [align=left] {\begin{minipage}[lt]{45.44pt}\setlength\topsep{0pt}
\begin{center}
{Recession fan}
\end{center}

\end{minipage}};
\end{tikzpicture}
\caption{A compactification-degeneration of the total space of the normal bundle of $Z$.}\label{fig: degeneration-fan}
\end{figure}
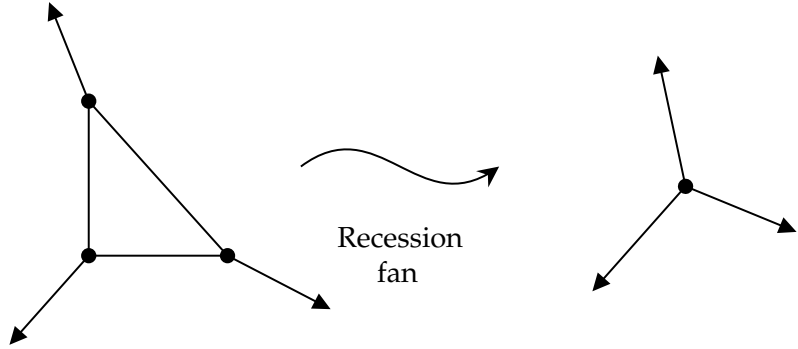

The result is a degeneration of $\mathbb P(\oplus L_i)$ into a degeneration of $\ell$ components meeting at $Z$. The normal bundle of $Z$ in this degeneration is exactly $\oplus_i L_i$, i.e. the same as its normal bundle in $\cY$. Let $\mathcal Z$ denote the full degeneration. It admits a contraction $\mathcal Z\to Z$. We can blow up the central fiber at this copy of $Z$ to introduce a copy of $E$ as above, with exactly the same multiplicity occurring in the proof above. Denote this new degeneration $\mathcal Z'$. 

We had a transverse sheaf on $Z$, namely the restriction of $\mathcal H_0$. We can pull back to $\mathcal Z$ along the contraction. We compute its Euler characteristic in three ways. The Euler characteristic of a sheaf pulled back along a projective space bundle is equal to the Euler characteristic of that sheaf on the base. This means that in the general fiber, we recover the Euler characteristic of $\mathcal H_0|_{Z}$. The Euler characteristics of the two special fibers, before and after blowup, must be equal because they are flat limits of a general fiber. Computing these by the inclusion/exclusion formula exactly then recovers the identity above. 
\end{remark}

\noindent
{\it Conclusion of Proof of Proposition~\ref{prop: toric-collapse}}. We have identified the restriction of $\mathcal G$ to the special fiber of $\mathcal Y\to\mathbb P^1$ with the known sheaf $\mathcal H_0$. We want to show that this implies that $\mathcal G$ is algebraically transverse, namely, flat over $[\mathcal Y/T]$. Since $\mathcal G'$ is flat over the $[\cY'/T]$, this check is only nontrivial in a neighborhood of $Z$. 

Flatness can be checked locally on the domain, so we can replace $\cY$ with an open set $U$ isomorphic to $\mathbb A^k\times\mathbb G_m^{r-k}$, and flatness over $[Y/T]$ is equivalent to flatness over $[\mathbb A^k/\mathbb G_m^k]$. We can shrink $U$ on $\cY$ until the map from $U$ to $[\mathbb A^k/\mathbb G_m^k]$ factorizes through $\mathbb A^k$. We will check flatness here. It will be useful to recall that the degeneration induced a map $\mathbb A^k\to\mathbb A^1$ given by the product of the coordinates, so the union of coordinate subspaces is the special fiber of this degeneration. 

Let $R$ denote the coordinate ring of this affine space $\mathbb A^k$ and fix coordinates $x_1,\ldots,x_k$ on this affine space. If we restrict to the closed point $(0,\ldots,0)$ in this $\mathbb A^k$ and pull back the sheaf $\mathcal G$, we obtain a sheaf over a point, which is certainly flat. 

Now consider the curve $\mathbb V(x_2,\ldots,x_k)$ and the stratum in $U$ that lies above it. Pull back $\mathcal G$ to this locus. We claim the restriction is flat over $\mathbb V(x_2,\ldots,x_k)$. To see this, we apply the local criterion for flatness~\cite[\href{https://stacks.math.columbia.edu/tag/00ML}{Tag 00ML}]{stacks-project} -- it suffices to check that multiplication by $x_1$ is injective on sections of $\mathcal G\otimes R/(x_2,\ldots,x_k)$. The restriction of $\mathcal G$ to $\mathbb V(x_2,\ldots,x_k)$ can be obtained by first restricting to the special fiber $\mathcal Y_0$ and then pulling back to $\mathbb V(x_2,\ldots,x_k)$. By Proposition~\ref{prop: candidate-is-special-fiber}, we see $\mathcal G\otimes R/(x_2,\ldots,x_k)$ is the restriction of $\mathcal H_0$ to some component of the special fiber, and then pulling back to $\mathbb V(x_2,\ldots,x_k)$. The restriction of $\mathcal H_0$ to any component is a transverse sheaf. So in particular, the multiplication by $x_1$ described above is injective on sections. We conclude that $\mathcal G\otimes R(x_2,\ldots,x_k)$ is flat over the axis $\mathbb A^1_{x_1}$. 

We continue inductively in this fashion. Since the restriction of $\mathcal G$ to any component of the special fiber is algebraically transverse, we conclude flatness over $\mathbb V(x_k)$ in this local picture. At the final step, we would like to know that multiplication by $x_k$ is injective on sections of $\mathcal G$ in this local neighborhood. This is immediate from the fact that $\mathcal G$ forms a flat family over $\mathbb A^1$ and $\mathbb V(x_k)$ is an open inside an irreducible component of the special fiber. Indeed, if any sections are supported on $\mathbb V(x_k)$, these sections are also supported over the origin in $\mathbb A^1$ contradicting flatness. The proof is complete. \qed  

\subsection{Proof in the affine case}  We prove the affine version of the result, stated in Proposition~\ref{prop: local-collapse}. We reset notation to $X = \mathbb A^r$ and $X'\to X$ a blowup a coordinate linear space $Z$. We have a sheaf $\mathcal F'$, given as the quotient of a trivial vector bundle, and $\mathcal F'$ is vertically trivial along $E\to Z$. Let $\mathcal F^\circ$ denote the restriction of this sheaf to the dense torus. 

Recall from Section~\ref{sec: pl-G-strat} that there is a stratification of the cocharacter space $N_\RR$ of the dense torus of $X$ that has the following property. For two points $v$ and $v'$ in the same stratum, if we blowup by introducing rays at these points, the restriction of $\mathcal F'$ to the torus interiors of the exceptional divisors give the same sheaf. We now make the following:

\noindent
{\bf Claim.} If $\mathcal Y_\Sigma$ is a smooth and projective toric compactification that refines the Gr\"obner stratification, then then the closure of $\mathcal F^\circ$, viewed as a quotient of a trivial sheaf of rank $n$, in $\mathcal Y_\Sigma$ is algebraically transverse. 

{\it Proof.} There certainly exists some smooth projective toric compactification such that the closure of $\mathcal F^\circ$ is algebraically transverse. The sequence of blowups of $\mathcal Y_\Sigma$ along smooth centers is cofinal in the system of all toric compactifications of a torus, so we can assume that there is a sequence of blowups
\[
\cY_N\to\cY_{N-1}\to\cdots\to \cY_1\to\cY_0 = \cY_\Sigma
\]
such that on $\cY_N$, the closure $\mathcal F_{N}$ of $\mathcal F^\circ$ is algebraically transverse. However, as all the fans of all $\cY_i$ necessarily refine the Gr\"obner stratification, the restriction of $\mathcal F_N$ to the exceptional divisor of $\cY_N\to\cY_{N-1}$ is vertically trivial. By applying Proposition~\ref{prop: toric-collapse} we conclude that it is the pullback of a transverse sheaf $\mathcal F_{N-1}$ on $\cY_{N-1}$. The claim follows by induction. \qed

We can finally deduce Proposition~\ref{prop: local-collapse}. Consider the Gr\"obner stratification of $\mathcal F^\circ$, defined in Section~\ref{sec: pl-G-strat}. The setup implies that the relative interior of each cone in the fan of $\mathbb A^r$ is contained in a single stratum of the Gr\"obner stratification. Now choose a smooth projective toric fan that refines the Gr\"obner stratification and contains this cone. This can be done by taking the common refinement of the Gr\"obner stratification with an arbitrary projective smooth toric compactification of $\mathbb A^r$. The cone corresponding to the original in $\mathbb A^r$ is unchanged. By iterated stellar subdivisions away from this cone, we find the required compactification. We deduce transversality of the closure of $\mathcal F^\circ$ on this toric compactification and obtain it on $\mathbb A^r$ by restriction. This concludes the affine local version of the statement. \qed

\subsection{The full statement} We deduce Theorem~\ref{thm: general-collapse}. We start with a smooth and projective simple normal crossings pair $(X,D)$ and let $p\colon X'\to X$ be its blowup at a stratum at $Z$. We have a quotient
\[
p^\star\mathcal E\twoheadrightarrow \mathcal F'
\]
that is vertical trivial along the exceptional $E$ over $Z$. By projectivity, we can replace $\mathcal E$ with some large number of copies of $\mathcal O_X(-n)$ for large $n$. The sheaf $\mathcal F$ is the closure of the restriction $X'\setminus E$ in $X$. Once we show that it is transverse the result will follows. To show this, observe that we can check transversality locally on $X'$. By shrinking, we can replace $X$ with an affine variety. By then picking generators, we can embed this affine variety into affine space. Finally, by including the local equations for $D$ in the generators and the pushing forward $\mathcal F$, we reduce the check of transversality to affine space itself, with respect to a subset of the coordinate boundary. The result follows from affine case. \qed

\section{$K$-tropicalizations and balancing}\label{sec: K-tropicalizations}


A point of the logarithmic Quot scheme $\mathsf{Quot}(X|D,\mathcal E)$ is a logarithmic surjection of sheaves, defined in~\cite{Logquot}. It is an equivalence class of quotients of pullbacks of $\mathcal E$ to degenerations of $X$, where two quotients are identified if they coincide on a ``further pullback''. The goal of this section is to define the {\it $K$-tropicalization} of such an object. It will be a polyhedral complex structure with polynomial decorations. We treat this in greater generality than we need -- in the end we will reduce to the toric case -- but believe that $K$-tropicalizations may be needed to study logarithmic coherent sheaves in later work.

\subsection{Tropical decorations for sheaves}\label{sec:TropDecorations} Essentially by definition, a logarithmic surjection can be represented on an {\it expansion} of $X$. One typically makes geometric arguments on these representatives.

Given a polyhedral complex $\Gamma$ decomposing $\Sigma_X$ into cells, the cone over it gives rise to a degeneration $\mathcal X_\Gamma\to \A^1$, by taking the constant family and subdividing. We assume that (i) the vertices are integer points, and (ii) the cone over $\Gamma$ is smooth as a cone complex.  We let $X_\Gamma$ denote the special fiber, with its logarithmic structure. There is also a relative Artin fan over $\A^1$, and we denote its special fiber by $\mathsf B_\Gamma$. 

Each component of $X_\Gamma$ inherits a simple normal crossings divisor coming from the union of the singular locus and the limit of $D$ in the special fiber. It also has a map
\[
q_\Gamma\colon X_\Gamma\to X.
\]
When the expansion is clear from context, we will use $q$, rather than $q_\Gamma$, to denote the map to $X$. We work with {algebraically transverse quotients of $\mathcal E$} on the expansion $X_\Gamma$, namely quotients $[q^\star\mathcal E\twoheadrightarrow \mathcal F]$ such that $\mathcal F$ is flat over $\mathsf B_\Gamma$. 

Recall that, if $X$ is a projective variety, $\mathcal F$ is a coherent sheaf on $X$, and $L_0,\ldots,L_d$ are line bundles, for a vector of integers $\underline n = (n_0,\ldots,n_d)$ we can consider the function
\[
P_X(F,L,\underline n) = \chi(X,\mathcal F\otimes\bigotimes_i L_i^{n_i}).
\]
This function is a multivariable polynomial in the $\underline n$. This follows from the Grothendieck--Riemann--Roch theorem. The polynomial is additive in short exact sequences. 

Given a pair $(X,D)$ we fix $L_0 = \mathcal O_X(1)$ to be any ample line bundle and let $L_i = \mathcal O_X(-D_i)$ be the ideal sheaf of $D_i$. 

\begin{definition}[$K$-weights]
Let $\mathcal F$ be an algebraically transverse sheaf on an expansion $X_\Gamma\to X$. Let $\gamma$ be a cell in $\Gamma$ and let $X_\gamma$ be the corresponding stratum of $X_\Gamma$. Its {\it $K$-weight}, denoted $k_\gamma(\underline t)$ of the cell $\gamma$ is the multivariable Hilbert polynomial, in variables $t_i$ and with respect to the line bundles $L_i$, of the sheaf $\mathcal F_\gamma:=\mathcal F|_{X_\gamma}$. 
\end{definition}

\noindent
{\bf $K$-tropicalization.} We defined the $K$-weights associated to an algebraically transverse sheaf on a particular expansion $X_\Gamma'\to X$. They descend to the Gr\"obner stratification as a consequence of Theorem~\ref{thm: canonical-transversalization} proved in the previous section, together with the facts about holomorphic Euler characteristics of bundles and blowups recorded in Section~\ref{sec: basic-formulae}. We refer to the Gr\"obner stratification together with its $K$-weights as the {\it $K$-tropicalization}. 

\begin{remark}
There are richer versions of the $K$-weights. Instead of just keeping track of the Hilbert polynomial, one can go further using $K$-theory of coherent sheaves on $X$ modulo numerical equivalence. We can record the function on $K_{\sf num}(X)$ obtained by pulling back classes from this numerical $K$-theory of $X$ tensoring with $\mathcal F_\gamma$ and taking Euler characteristic.
\end{remark}

When $(X,D)$ is a toric pair, there is significant redundancy in the $K$-weights. Indeed, because the toric boundary generates the Picard group, any line bundle is represented by an invariant line bundle, and the data is essentially equivalent to the constant terms of the $K$-weights The following definition is therefore much cleaner in this case.

\begin{definition}[Toric $K$-weights]
Let $X$ be a toric variety and $D$ the full toric boundary divisor. Let $\mathcal F$ be an algebraically transverse sheaf on an expansion $X_\Gamma\to X$. Let $\gamma$ be a cell in $\Gamma$ and let $X_\gamma$ be the corresponding stratum of $X_\Gamma$. The {\it toric $K$-weight}, or simply {\it $K$-weight} when the context is clear, is the Euler characteristic of the sheaf $\mathcal F_\gamma:=\mathcal F|_{X_\gamma}$. 
\end{definition}

In order to control the discrete data of the logarithmic Quot space, it will be useful to explain how $K$-weights behave under the operation of taking asymptotic (also known as {\it recession}) fan. 

As above, let $\Gamma$ be a rational polyhedral decomposition of $\Sigma$ with integer vertices. The {\it asymptoic fan} of $\Gamma$, denoted $\Gamma_0$, is the $0$ fiber of the cone
\[
\mathsf{Cone}(\Gamma)\to\mathbb R_{\geq 0},
\]
see for example~\cite{BGS11,NishSieb06}. Fix a polynomial weighting
\[
k\colon \Gamma \to\mathbb Q[\underline t].
\]
We view $\Gamma$ here as a set of polyhedra, rather than as its topological realization. 
The set of faces is a partially ordered set by inclusion. Every face $F$ of $\Gamma$ has a dimension, which we denote ${\sf dim} \ F$. There is a specialization map on face posets:
\[
\mathsf{sp}\colon \Gamma\to \Gamma_0. 
\]

\begin{definition}[Asymptotic $K$-weights]\label{def: asymptotic-weights}
Given a polyhedral complex $\Gamma$ with a weight function $w$ on the faces as above, the {\it asymptotic $K$-weights} on $\Gamma_0$ are defined by the following formula: for a face $G$ in $\Gamma_0$, define
\[
k_0(G) = \sum_{F\in \mathsf{sp}^{-1}(G)} (-1)^{\mathsf{dim}(F)-\mathsf{dim}(G)}\cdot k(F).
\]
In words, we sum the decorations of the faces of the same dimension as $G$ that specialize to it, subtract those of one dimension higher, add those of two higher, and so on. 
\end{definition}

Geometrically, the asymptotic $K$-weights arise in families. First note that if there is specialization of $(X,D)$ to an expansion $X_\Gamma$, then the recession fan of $\Gamma$ is necessarily the fan of $X$. The asymptotic weights control the numerical data in such specializations. Precisely, if there is a flat specialization of an algebraically transverse sheaf on $X$ to one on an expansion, the asymptotic $K$-weights of the special fiber compute the $K$-weights of the general fiber. The asymptotic $K$-weights are constant in flat families.

\subsection{$K$-balancing} The tropicalization of a sheaf with its $K$-weights is constrained by a balancing condition. It is using this balancing condition that we will prove our main combinatorial result -- the set of all tropicalizations of quotients of a given sheaf with fixed numerical data is finite. 

\subsubsection{Warm up -- the source of balancing}\label{sec: ordinary-balancing} We explain the basic geometry behind the balancing condition, without explicit formulae, before giving full details, starting with the most well-studied case~\cite{IntToricVar,MaclaganSturmfels,Mi03}.

Let $X$ be a smooth and projective toric variety and $Z\hookrightarrow X$ a closed subvariety of dimension $e$ that meets the strata of $X$ in expected dimension. Let $B^e(X)$ denote the set of strata of codimension $e$. Intersection with $Z$ gives a function
\begin{eqnarray*}
\cap Z\colon \mathsf B^e(X)&\to& \mathbb Z\\
W&\mapsto& \#(W\cap Z),
\end{eqnarray*}
where $\#$ denotes the length of the intersection. The balancing condition asserts that this function is constrained. Geometrically, the count $ \#(W\cap Z)$ is the degree of the intersection product in Chow: $\#(W\cap Z) = \int_X [W]\cdot [Z]$. If we extend $\cap Z$ linearly to the free abelian group $\mathsf Z^e(X)$ on $B^e(X)$, it descends to the Chow group:
\[
\mathsf{Z}^e(X)\twoheadrightarrow \mathsf{CH}^e(X).
\]
In particular, if $\sum a_i W_i$ is a cycle that represents $0$ in Chow (or even in homology, though this is equivalent) then we have a constraint:
\[
\sum a_i\cdot \#(Z\cap W_i) = 0.
\]
To obtain formulae from this, one notes that a basis of rational equivalences can be obtained by divisors associated to characters on dense tori of each stratum of $X$ of codimension $e-1$. 

We want to generalize replace $Z$ be an algebraically transverse coherent sheaf $\mathcal F$ and then generalize this in two ways:
\begin{enumerate}[(i)]
\item We would like to record the Euler characteristic (or similar invariants, such as the Hilbert polynomial) of intersections with strata.
\item We would like to replace the toric variety with a general snc pair. 
\end{enumerate}

For (i), the basic geometric idea is the same. We can record Euler characteristics of $\mathcal F$ with {\it all} strata of $X$. If  $\mathcal B(X)$ denotes the set of strata we consider:
\begin{eqnarray*}
\otimes \mathcal F\colon \mathcal B(X)&\to& \mathbb Z\\
W&\mapsto& \chi(X, \mathcal F \otimes \mathcal O_W)
\end{eqnarray*}
The line bundle associated to a character $\chi$ is trivial, and gives an equality of certain ideal sheaves. The associated equality of structure sheaves constrains the function.

For (ii), our strategy is to realize $(X,D)$ as a subvariety of a toric variety bundle over $X$, and use formulae for the $K$-theory ring of a toric bundle. Again, we note that we will eventually only use $K$-balancing in the toric setting, so a casual reader might be advised to  skip or skim this construction. 

\subsubsection{Embedding into a toric bundle} Let $D_1,\ldots, D_k$ be the divisor components of $D$, and let $(L_i,s_i)$ be the associated line bundle/section pair. The ideal sheaf of $D_i$ is $L_i ^\vee$. Consider the projective space bundle
\[
\mathbb P = \mathbb P_X\left(\bigoplus_i L_i\oplus \mathcal O\right)\to X.
\]
The space $\mathbb P$ has a natural simple normal crossings divisor, given by the fiberwise toric boundary, with no logarithmic structure coming from $X$ itself. 

The sections $s_i$ give rise to a section
\[
s\colon X\to \mathbb P
\]
The divisor $D\subset X$ is precisely the pullback $s^{-1}(\partial \mathbb P)$. Therefore the map
\[
s\colon (X,D)\hookrightarrow (\mathbb P,\partial \mathbb P)
\]
is strict\footnote{We warn the reader that there are two logarithmic structures in play on $X$ -- the trivial logarithmic structure which is the correct structure on the base of the bundle $\mathbb P$, and the snc divisor that we started with, which is the pullback of the fiberwise toric structure under the embedding $s$.}. 

The next lemma guarantees that we may as well replace $(X,D)$ with $(\mathbb P,\partial \mathbb P)$ for the purposes of boundedness.We maintain the notation and setup above. 

\begin{lemma}
Let $\mathcal E$ be a coherent sheaf on $X$ that is flat over $\mathsf A(X,D)$. The sheaf $s_\star\mathcal E$ is flat over $\mathsf A(\mathbb P,\partial \mathbb P)$. 
\end{lemma}

\begin{proof}
Immediate from strictness, as $\mathsf A(X,D)$ is an open substack of $\mathsf A(\mathbb P,\partial \mathbb P)$. 
\end{proof}

\begin{lemma}\label{lem:CanPullBack}
Every expansion of $X$ along $D$, possibly after further expansion, is the pullback of an expansion of $\mathbb P$ along $\partial \mathbb P$, via the map $s$. 
\end{lemma}

\begin{proof}
An expansion of $X$ is given by a polyhedral decomposition of $\Sigma(X,D)$. This is a face of $\Sigma(\mathbb P,\partial \mathbb P)$. We can certainly choose a polyhedral decomposition of $\Sigma(\mathbb P,\partial \mathbb P)$ that refines the given one on $\Sigma(X,D)$. One way to see this is to express the subdivision of the face as a sequence of stellar subdivisions. These extend by elementary considerations. 
\end{proof}

\begin{remark}
    The phrase \textit{possibly after expansion} in Lemma \ref{lem:CanPullBack} is not necessary. We omit a proof of this stronger statement as we will not need it, and the proof is more finicky.
\end{remark}

If $\mathcal F$ is a transverse sheaf on an expansion $X_\Gamma$ its pushforward $s_\star\mathcal F$ is transverse on $\mathbb P_\Gamma$, the $K$-weights of the tropicalization of $\mathcal F$ on $X_\Gamma$ coincide with those of the pushforward $s_\star\mathcal F$ with respect to the pullback of $L$ to $\mathbb P$. So for the purposes of boundedness, we can replace $(X,D)$ with $(\mathbb P,\partial \mathbb P)$.

\subsubsection{Warming up again -- cycle theory in toric bundles} To explain how the $L_i$ affect balancing, we first explain balancing at the level of cycles, before moving on to sheaves. For ease of exposition we focus on balancing for tropicalization associated to curves embedded in bundles of arbritary dimension. The material here is essentially lifted from~\cite{CMN,Dod24}. 

Let $p\colon \mathbb P\to X$ be as above and let $Z\hookrightarrow \mathbb P$ be a curve, meeting the fiberwise toric boundary in expected dimension. We obtain a function on the strata:
\begin{eqnarray*}
B^1(\mathbb P)&\to&\mathbb Z\\
D&\mapsto& \#(D\cap Z),
\end{eqnarray*}
where $\#$ denotes the length of the intersection. Balancing is a constraint on this function. 

More precisely, let $\beta$ be $p_\star[Z]$. Let $u$ be an element of the character lattice $M$ of the fiber torus. By standard toric geometry, we obtain a divisor
\[
\mathsf{div}(u) = \sum_{\rho} \langle u,v_\rho\rangle E_\rho
\]
where $\rho$ ranges over the rays of the fan of the fiber torus and $v_\rho$ is a primitive generator for the ray. If the line bundle factors $L_i$ were all trivial, then $\mathsf{div}(u)$ would be the trivial divisor. 

In general, the divisor $\mathsf{div}(u)$ is the pullback of a line bundle $L_i$ on $X$. For an appropriate basis $u_1,\ldots,u_k$, we can arrange for
\[
\mathsf{div}(u_i) \cong p^\star L_i
\]
and the $L_u$ are determined by linearity. We must have the modified balancing condition:
\[
\sum_{\rho} \langle u,v_\rho\rangle \#(E_\rho\cap Z) = \beta\cdot c_1(L_u). 
\]

\begin{example}
As a concrete example, let $X = \mathbb P^1$ and $\mathbb P$ be the projective line bundle
\[
\mathbb P(\mathcal O(f)\oplus\mathcal O)\to \mathbb P^1. 
\]
The divisor is the union of $E_0$ and $E_\infty$ be the $0$ and $\infty$ sections of the bundle. 

Let $Z\subset \mathbb P$ be a curve. The transversality condition ensures that $Z$ meets the sections in dimension $0$. The class $\beta$ is a positive integer, which counts the intersection number of $Z$ with a fiber. Let $n_0$ and $n_\infty$ be the respective intersection numbers. We have the relation
\[
[E_0]-[E_\infty] = \pi^\star c_1(\mathcal O(f)).
\] 

\end{example}

We explain via the next example how the picture above changes for subvarieties of dimension $2$. 

\begin{example} Consider $\pi \colon \mathbb P\to X$ a $\mathbb P^1$-bundle with associated line bundle $L$, and $Z\hookrightarrow X$ a subvariety of dimension $2$. Let $E_0$ and $E_\infty$ be the $0$ and $\infty$ sections of the bundle. Then we have
\[
[E_0]-[E_\infty] = \pi^\star c_1(L).
\] 
Let $H$ be a generic hyperplane section of $X$. Capping the curves $Z\cap E_0$ and $Z\cap E_\infty$ with $\pi^\star H$ we obtain $0$-dimensional schemes with lengths $n_0$ and $n_\infty$ respectively. The modified balancing condition reads
\[
n_0-n_\infty = \beta\cdot f,
\]
while ordinary balancing would assert $n_0-n_\infty$ is $0$. 
\end{example}
Note the balancing condition depends on $f$, and thus the bundle. Fixing $\beta$, when the bundle is trivial, i.e. $f = 0$, we see that $n_0$ and $n_\infty$ are the same and we obtain traditional balancing. As $f$ changes, the difference between $n_0$ and $n_\infty$ changes, but in a manner that is completely dictated by $\beta$ and the Chern class of the bundle. If $\beta$ is $0$, i.e. if $Z$ is a fiber class, then again we have traditional balancing. 

\subsubsection{$K$-theory in toric bundles} We formulate the $K$-theory balancing using results of Sankaran--Uma~\cite{SU03} that describe, among many other things, the $K$-theory of toric fiber bundles. 

We alter the notation and allow $\mathbb P$ to be {\it any} equivariant smooth projective compactification of the torus bundle 
\[
\bigoplus_i L_i^\star \to X
\]
associated to a fan $\Sigma$. Let $N$ denote the fiber cocharacter lattice. 

The splitting $\bigoplus L_i$ gives a natural isomorphism $N\cong\mathbb Z^k$. If $u\in M$ is an element of the dual, we associate a line bundle
\[
L_u = \bigotimes_i L_i^{u_i}
\]
where $u_i$ is the $i$-component of $u$. Finally, let $\Sigma$ be the fan of the fiber. We use $\rho_i$ for the rays and $v_i$ for their primitive generators. We use $E_{\rho_i}$ the divisor associated to the ray $\rho_i$. The $K$-theory of vector bundles on $\mathbb P$ and the $K$-theory of coherent sheaves are naturally isomorphic, since $\mathbb P$ is assumed to be smooth and projective. It is given as follows. 

\begin{theorem}[Sankaran--Uma]
Let $\mathbb P\to X$ be as above. The Grothendieck $K$-ring of $\mathbb P$ is isomorphic to a quotient of 
\[
K(X)[x_1,\ldots,x_k]
\]
subject to the relations:
\[
\prod_{i\in I} x_i = 0, \ \ \textnormal{if the rays $\rho_i$ for $i\in I$ do not form a cone},
\]
and for each $u\in M$
\[
\prod_{i|\langle v_i,u\rangle>0} (1-x_i)^{\langle v_i,u\rangle} = [L_u]\cdot \prod_{j|\langle v_j,u\rangle<0} (1-x_i)^{-\langle v_j,u\rangle}.
\]
Under the isomorphism, $x_i$ is mapped to the structure sheaf $\mathcal O_{E_i}$, so $(1-x_i)$ is the ideal sheaf $\mathcal O_{\mathbb P}(-E_i)$. 
\end{theorem}

The toric case plays a special role. In this case, the ring is obtained by replacing $K(X)$ with $\mathbb Z$ and setting all the $L_u$ to $1$. In other words, the $K$-groups are spanned by structure sheaves of torus invariant subvarieties. The key relation for us is the second one, which states that the divisor of a character defines $1$ in $K$-theory. 

\subsubsection{Balancing for coherent sheaves} We formulate general $K$-balancing for sheaves on toric bundles $p\colon \mathbb P\to X$. A reader only interested in the toric case can skip to the next section without loss of continuity. 

Let $\mathcal F$ be an algebraically transverse sheaf on a toric bundle $\mathbb P$ with fiber fan $\Sigma$. Recall we have fixed line bundles $L_0,\ldots, L_d$ on the base $X$ of the projective bundle, where $L_0$ is a reference ample on $X$. 

The $K$-weights are multivariable polynomials that record, for each cone $\sigma$ in $\Sigma$ the Hilbert polynomial
\[
\sigma\mapsto P(\mathcal F|_{W_\sigma},\underline L,\underline t).
\]
By construction, this polynomial records the Euler characteristic of every twist by the $L_i$'s of every restriction to $\mathcal F$ to a stratum. 

Our goal is to explain constraints on the possible polynomial decorations that could arise in this way. Fix a {\it polynomial decoration} given by any function
\[
\Sigma\to \mathbb Z[\underline t].
\]
Any $K$-theory class on $\mathbb P$ gives a decoration of this kind. Indeed, if $\alpha$ is such a class, then we can pull back to any stratum $W_\sigma$, push forward to $X$, and then pass to multigraded Hilbert polynomial with respect to the system $\underline L$. We refer to this as the {\it the associated polynomial decoration} of $\alpha$. 

\begin{definition}
A polynomial decoration
\[
\Sigma\to \mathbb Z[\underline t].
\]
{\it satisfies the $K$-balancing condition} if it arises as the associated polynomial decoration of a $K$-theory class $\alpha\in K(\mathbb P)$. 
\end{definition}

\subsubsection{$K$-balancing as a sequence of tests} The statement above heavily constrains the possible polynomial decorations, but is somewhat inexplicit, and necessarily more subtle than traditional balancing. We now unpack this to something more checkable. 

The polynomial decoration is a function associated to reduced and irreducible closed strata of $\mathbb P$. Choose a basis $e_1,\ldots,e_d$ for the character lattice $M$ of the fiber torus, such that
\[
[\mathsf{div}(e_i)] = [p^\star \mathcal O_X(D_i)]. 
\]
Correspondingly for $u\in M$, we see that $\mathsf{div}(u)$ is equal to the pullback of a tensor power $\bigotimes_i \mathcal O_X(a_iD_i)$. 

Recall that $\rho$ denotes the set of rays in $\Sigma$, the primitive generators are $v_\rho$, and we use $\mathcal O(-\rho)$ for the ideal sheaf of $E_\rho$. The relation determined by $u$ is given by:
\[
\prod_{\rho|\langle v_\rho,u\rangle >0} \mathcal O(-\rho)^{\langle v_\rho,u\rangle} = p^\star L_u\cdot \prod_{\rho|\langle v_\rho,u\rangle <0} \mathcal O(-\rho)^{-\langle v_\rho,u\rangle}
\]
We can rewrite $L_u$ as a tensor product of ideal sheaves and their duals, and then subtracting both sides from $1$ we obtain a relation:
\[
1-p^\star L_u^{+}\cdot \prod_{\rho|\langle v_\rho,u\rangle >0} \mathcal O(-\rho)^{\langle v_\rho,u\rangle} = 1-p^\star L^{-}_u\cdot \prod_{\rho|\langle v_\rho,u\rangle <0} \mathcal O(-\rho)^{-\langle v_\rho,u\rangle}
\]
Subtracting both sides from $1$, we obtain an equality of structure sheaves of a union of divisors in $\mathbb P$ with regular crossings. 

Given a polynomial decoration
\[
\Sigma\to \mathbb Z[\underline t].
\]
we can {\it formally} compute the left and right hand sides of the relation above. Precisely, the class
\[
1-p^\star L_u^{+}\cdot \prod_{\rho|\langle v_\rho,u\rangle >0} \mathcal O(-\rho)^{\langle v_\rho,u\rangle}
\]
is represented by a divisor with multiple, possibly non-reduced, components with regular crossings -- the divisors $E_\rho$ appear with multiplcity given by $\langle v_\rho,u\rangle$ and the components $D_i\subset X$ appear with multiplicity as given by the entries in $u$. The class can therefore be computed by inclusion exclusion of non-reduced complete intersection strata in $\mathbb P$, and intersections of these with pullbacks of $a_iD_i\subset X$. 

By using the discussion of multiplicities in Section~\ref{sec: toolbox}, the polynomial decorations associated to these non-reduced strata are determined formally. We can similarly compute the right hand side. For each $u$, this gives a test for balancing. 

Note that the calculation of the decorations on these non-reduced strata requires a choice. By Section~\ref{sec: toolbox}, this involves computing the decoration associated to derived self-intersections of divisors $E_\rho$. By toric geometry, these self-intersections can be expressed as higher codimension strata of $\mathbb P$, with possible corrections coming from pullbacks from the base -- the choice of higher codimension strata is the choice. The fact that different choices give the same answer is itself a test of balancing!

The balancing condition must be tested for all, or equivalently a basis, of elements of $M$. Furthermore, for each cone in $\Sigma$, we can pass to the star and again check the balancing condition. 

All these tests, together, constitute the $K$-balancing condition. 

We will not use the actual formulae, as our main proof will reduce to the toric case, so we end this discussion here. 

\subsubsection{The toric specialization} We restate balancing in the toric case. Let $(Y,E)$ be a toric variety with character lattice $M$. Given a transverse sheaf $\mathcal F$ on $Y$, consider the decoration:
\[
\sigma\mapsto \chi(Y,\mathcal O_{W_\sigma}\otimes \mathcal F).
\]
Since $\mathcal F$ is transverse, this Euler characteristic is obtained by taking the product $[\mathcal O_{W_\sigma}]\cdot[\mathcal F]$ in the $K$-theory of $Y$, and pushing forward to a point. Exactly as in the Chow section explained in Section~\ref{sec: ordinary-balancing}, because this can be computed in $K$-theory, there are constraints on the function. 

Consider an integer-valued decoration
\[
\Sigma\to \mathbb Z.
\]
Given a $K$-theory class $\alpha$ on $Y$, we obtain, {\it the associated decoration of $\alpha$} a decoration by
\[
\sigma\mapsto \chi(Y,\alpha\cdot[\mathcal O_{W_\sigma}]). 
\]

\begin{definition}
An integer decoration
\[
\Sigma\to \mathbb Z.
\]
{\it satisfies the $K$-balancing condition} if it arises as the associated decoration of a $K$-theory class $\alpha\in K(Y)$. 
\end{definition}

Specializing the discussion in the previous section, practically we check this as follows. For each linear function $u\in M$, we obtain an equality of $K$-theory classes, in fact of structure sheaves of divisors
\[
1-\prod_{\rho|\langle v_\rho,u\rangle>0} \mathcal O(-\rho)^{\langle v_\rho,u\rangle} = 1-\prod_{\rho|\langle v_\rho,u\rangle<0} \mathcal O(-\rho)^{-\langle v_\rho,u\rangle}.
\]
By the discussion in Section~\ref{sec: toolbox}, there is a universal formula expressing this in terms of $K$-theory classes of strata. 

To test balancing around the origin in $\Sigma$, given a decoration $\Sigma\to \mathbb Z$, for each $u$, we check that the left and right hand sides of the relation above are satisfied. We similarly check the analogous relation around every cone in $\Sigma$ by passing to the star. 

\subsection{Manipulating toric $K$-balancing} The material in this section concerns the {\it toric} $K$-balancing condition, rather than the general case, as this will suffice for our boundedness argument. 

The discussion in the previous section provides a condition for a fan $\Sigma$ with a decoration
\[
k\colon \Sigma\to \mathbb Z
\]
to be {\it $K$-balanced}. Given an algebraically transverse sheaf $\mathcal F$ on an expansion $Y_\Gamma$ of a toric pair $(Y,E)$, we obtain a weight function on the polyhedral complex $\Gamma$:
\[
k\colon \Gamma\to \mathbb Z
\]
by restricting $\mathcal F$ to each corresponding closed stratum of $Y_\Gamma$ and taking Euler characteristic. 

\begin{definition}
A polyhedral complex $\Gamma$ with a decoration
\[
k\colon \Gamma\to\mathbb Z
\]
is {\it $K$-balanced} if the induced decoration on the star fan of every face is $K$-balanced. 
\end{definition}

The $K$-balancing condition, even in the toric case, has high combinatorial complexity. We extract some checkable consequences in the next two sections that capture the finiteness we need. First, we construct a filtration on a $K$-balanced polyhedral complex that, in some sense, mirrors the torsion filtration on a coherent sheaf on a scheme. Next, we explain how the constraint behaves under linear projections. 

\subsubsection{The tropical dimension filtration} Fix a polyhedral complex $\Gamma$ and an integer decoration
\[
k\colon \Gamma\to \mathbb Z
\]
that satisfies the $K$-balancing condition. We say the {\it dimension} of $(\Gamma,k)$ is the maximal dimension of a cell on which $k$ takes a nonzero value. 

The balancing condition is a condition on the star of any face of $\Gamma$. The condition is trivial about faces of dimension larger than or equal to $\dim (\Gamma,k)$. The first nontrivial balancing condition occurs in codimension $1$. The star of a codimension $1$ cell is a tropical curve, and here the condition is the traditional balancing condition. 

\begin{definition}[Pullback under subdivision]
Let $\Gamma'\to \Gamma$ be a subdivision. The {\it pullback $k'$ of the weight function $k$} is defined by the composite
\[
\Gamma'\to\Gamma\xrightarrow{k}\mathbb Z.
\]
\end{definition}

The pullback weight is geometrically meaningful. 

\begin{lemma}
If $\mathcal F$ is an algebraically transverse sheaf on an expansion $Y_\Gamma$ of $Y$, and determines a $K$-weight function $k$ on $\Gamma$, if $Y'\to Y$ is a modification associated to a subdivision $p\colon \Gamma'\to\Gamma$, the decoration determined by $p^\star\mathcal F$ is the pullback of $k$. 
\end{lemma}

\begin{proof}
Essentially identical to Lemma~\ref{lem: blowup-formula}.
\end{proof}

The notion of pullback allows for a natural ``coarsening'' of a fan with $K$-weights. For this it is natural to work in the category of {\it piecewise linear spaces} from~\cite{Logquot}, recalled in Section~\ref{sec: pl-G-strat}. For us, this is simply a stratification of $N_{\mathbb R}$ with linear strata that admits a refinement by a fan. 

Piecewise linear spaces can be equipped with integer-valued weight functions. An {\it integer decorated} piecewise linear space is a piecewise linear space equipped with an integer function on the set of strata. It is {\it $K$-weighted} if some fan structure refining it, equipped with the pullback weights, satisfies the $K$-balancing condition. Given two $K$-weighted piecewise linear spaces stratifying $N_\RR$, a morphism between them is a refinement of stratifications equipped with the pullback weight. 

\begin{definition}
The {\it canonical coarsening of $(\Gamma,k)$} is the terminal object in the category of $K$-weighted piecewise linear subdivisions of $N_\RR$, with morphisms given above. A $K$-weighted piecewise linear space is {\it minimal} if it is equal to its own canonical coarsening. 
\end{definition} 

\begin{lemma}
Canonical coarsenings exist.
\end{lemma}

\begin{proof}
Let $(\Gamma,k)$ be a weighted piecewise linear space. We define a locally closed stratification $\mathcal{P}$ of $N_\mathbb{R}$ indexed by the range of $k$. A point $p$ lies in stratum with index $x$ if $p$ lies in the interior of a cone $\sigma$ for which $k(\sigma) = x$. Such a locally closed stratification defines a piecewise linear subdivision of $N_\mathbb{R}$ with the required properties by \cite[Proposition 1.2.5]{Logquot}.
\end{proof}

\begin{example}
We might visualize the canonical coarsening as in Figure~\ref{fig: canonical-coarsening}. In the category of piecewise linear spaces, one can ``glob together'' cells across which the $K$-weights do not change. In the figure, the labels indicate the $K$-weights. In the right hand side figure, the interior of the triangle has two strata -- the complement of the point labelled with $7$, and the rest of the interior. The edges and vertices form the remaining vertices. 

\begin{figure}[h!]
\begin{tikzpicture}[scale=1]

\coordinate (A1) at (0,0);
\coordinate (B1) at (6,0);
\coordinate (C1) at (0,6);
\coordinate (M1) at ($(A1)!1/3!(B1)!1/3!(C1)$);

\coordinate (A2) at (10,0);
\coordinate (B2) at (16,0);
\coordinate (C2) at (10,6);
\coordinate (M2) at ($(A2)!1/3!(B2)!1/3!(C2)$);


\draw[thick, black] (A1) -- (B1) node[midway, below, black] {\textbf{5}};
\draw[thick, black] (B1) -- (C1) node[midway, right, black] {\textbf{5}};
\draw[thick, black] (C1) -- (A1) node[midway, left, black] {\textbf{5}};
\draw[thick, black, dashed] (M1) -- (A1) node[midway, left, black] {\textbf{2}};
\draw[thick, black, dashed] (M1) -- (B1) node[midway, below right, black] {\textbf{2}};
\draw[thick, black, dashed] (M1) -- (C1) node[midway, above, black] {\textbf{2}};

\foreach \p in {A1, B1, C1, M1}{
    \fill[black] (\p) circle (1.5pt);
    \node[above right, scale=1] at (\p) {\textbf{7}};
}

\node[black] at ($(A1)!0.5!(M1)!0.5!(C1)$) {\textbf{2}};
\node[black] at ($(A1)!0.5!(M1)!0.5!(B1)$) {\textbf{2}};
\node[black] at ($(B1)!0.5!(M1)!0.5!(C1)$) {\textbf{2}};

\draw[thick, black] (A2) -- (B2) node[midway, below, black] {\textbf{5}};
\draw[thick, black] (B2) -- (C2) node[midway, right, black] {\textbf{5}};
\draw[thick, black] (C2) -- (A2) node[midway, left, black] {\textbf{5}};

\foreach \p in {A2, B2, C2, M2}{
    \fill[black] (\p) circle (1.5pt);
    \node[above right, scale=1] at (\p) {\textbf{7}};
}

\draw[thick, ->, >=stealth, line width=2pt] (7.5,3) -- (9.5,3);

\node at ($(A2)!0.52!(B2)!0.28!(C2) + (0.2,0.2)$) {\textbf{2}};;
\end{tikzpicture}
\caption{An example of the canonical coarsening procedure.}\label{fig: canonical-coarsening}
\end{figure}
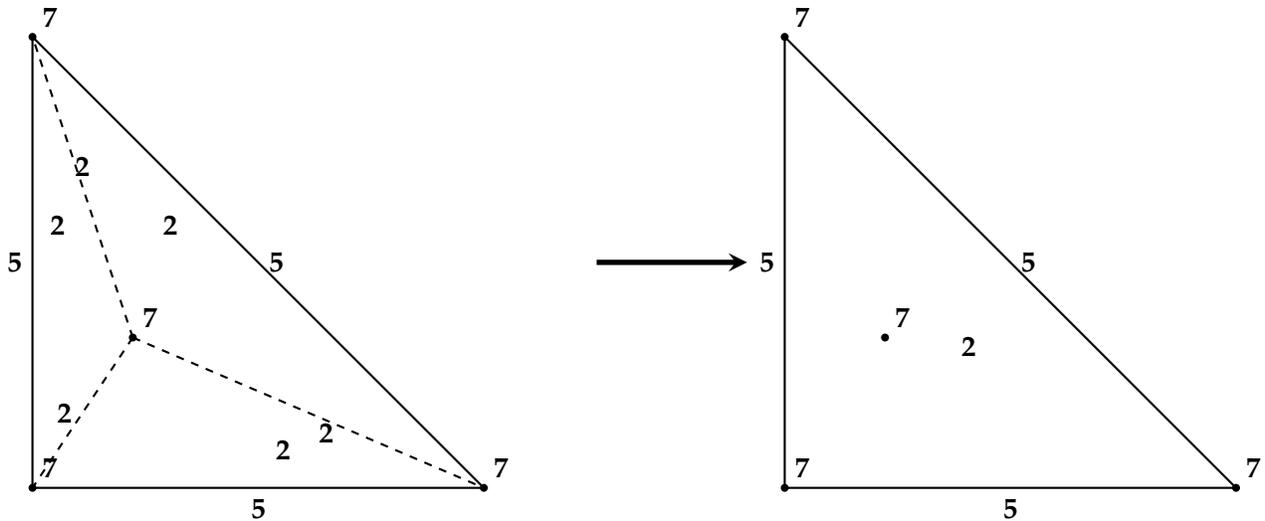

\end{example}

This notion allows us to define a natural filtration on $(\Sigma,k)$.  

\begin{definition}
Given a $K$-weighted piecewise linear space $(\Gamma,k)$ its {\it truncation below dimension $r$} is constructed as follows. First adjust the weights according to the following rule. For cells of dimension at least $r$, change nothing. For a cell $\gamma$ that has dimension smaller than $r$, change it in the following situation -- if every stratum of $\Gamma$ that properly contains $\gamma$ in its closure has the same weight, say $w$, then change the weight of $\gamma$ to $w$. Now pass to canonical coarsening.

The truncation below dimension $r$ is denoted $(\Gamma,k)_{\geq r}$, or simply $\Gamma_{\geq r}$ when $k$ is clear from context. 
\end{definition}

\begin{example}
The next three diagrams shown in Figure~\ref{fig: dim-filt} capture the tropical dimension filtration -- the first picture is the $K$-weighted piecewise linear space; the weights are omitted. The next two are its truncations below dimension $1$ and below dimension $2$. 
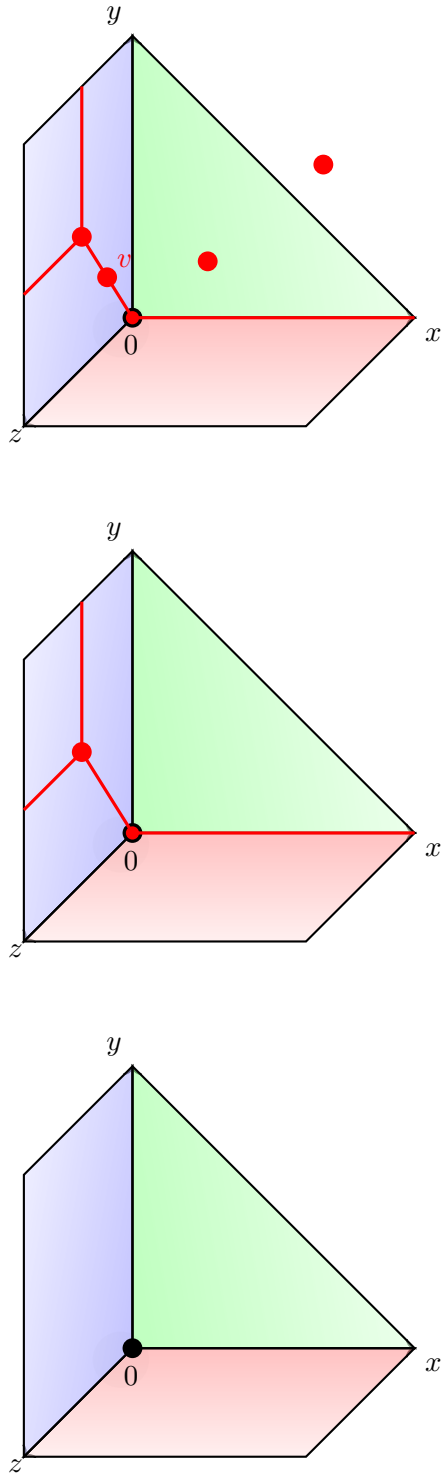
\begin{figure}[h!]
$$
\begin{tikzpicture}[scale=2.5]

\coordinate (O) at (0,0,0);

\coordinate (X) at (1.5,0,0);
\coordinate (Y) at (0,1.5,0);
\coordinate (Z) at (0,0,1.5);

\draw[thick,->] (O) -- (X) node[below right] {$x$};
\draw[thick,->] (O) -- (Y) node[above left] {$y$};
\draw[thick,->] (O) -- (Z) node[above right,xshift=-10pt,yshift=-10pt] {$z$};

\fill[gray!40,opacity=0.15] ($(O)+(-0.1,-0.1,-0.1)$) circle (0.15);

\shade[bottom color=blue!10, top color=blue!40, shading angle=-100, opacity=0.6]
  (O) -- (Y) -- ($(Y)+(0,0,1.5)$) -- (Z) -- cycle;

\shade[bottom color=red!10, top color=red!40, shading angle=0, opacity=0.6]
  (O) -- (X) -- ($(X)+(0,0,1.5)$) -- (Z) -- cycle;

\shade[bottom color=green!10, top color=green!40, shading angle=100, opacity=0.6]
  (O) -- (X) -- (Y) -- cycle;

\draw[thick] (O) -- (X) -- ($(X)+(0,0,1.5)$) -- (Z) -- cycle;
\draw[thick] (O) -- (Y) -- ($(Y)+(0,0,1.5)$) -- (Z) -- cycle;
\draw[thick] (O) -- (X) -- (Y) -- cycle;

\fill[black] (O) circle (1.5pt);
\node[below left,xshift=6pt,yshift=-3pt,black] at (O) {$0$};

\fill[red] (0.4,0.3,0) circle (1.5pt);
\fill[red] (1.4,1.2,1.0) circle (1.5pt);

\coordinate (V) at (0,0.7,0.7);

\coordinate (Vy) at (0,0.7,1.5);   
\coordinate (Vz) at (0,1.5,0.7);   

\coordinate (M) at ($(V)!0.5!(O)$);

\draw[very thick, red] (V) -- (O);
\draw[very thick, red] (V) -- (Vy);
\draw[very thick, red] (V) -- (Vz);

\fill[red] (V) circle (1.5pt);

\fill[red] (M) circle (1.5pt);
\node[above right, red] at (M) {$v$};

\draw[very thick, red] (O) -- (X);
\fill[red] (O) circle (1pt);

\end{tikzpicture}
$$

$$\begin{tikzpicture}[scale=2.5]

\coordinate (O) at (0,0,0);

\coordinate (X) at (1.5,0,0);
\coordinate (Y) at (0,1.5,0);
\coordinate (Z) at (0,0,1.5);

\draw[thick,->] (O) -- (X) node[below right] {$x$};
\draw[thick,->] (O) -- (Y) node[above left] {$y$};
\draw[thick,->] (O) -- (Z) node[above right,xshift=-10pt,yshift=-10pt] {$z$};

\fill[gray!40,opacity=0.15] ($(O)+(-0.1,-0.1,-0.1)$) circle (0.15);

\shade[bottom color=blue!10, top color=blue!40, shading angle=-100, opacity=0.6]
  (O) -- (Y) -- ($(Y)+(0,0,1.5)$) -- (Z) -- cycle;

\shade[bottom color=red!10, top color=red!40, shading angle=0, opacity=0.6]
  (O) -- (X) -- ($(X)+(0,0,1.5)$) -- (Z) -- cycle;

\shade[bottom color=green!10, top color=green!40, shading angle=100, opacity=0.6]
  (O) -- (X) -- (Y) -- cycle;

\draw[thick] (O) -- (X) -- ($(X)+(0,0,1.5)$) -- (Z) -- cycle;
\draw[thick] (O) -- (Y) -- ($(Y)+(0,0,1.5)$) -- (Z) -- cycle;
\draw[thick] (O) -- (X) -- (Y) -- cycle;

\fill[black] (O) circle (1.5pt);
\node[below left,xshift=6pt,yshift=-3pt,black] at (O) {$0$};


\coordinate (V) at (0,0.7,0.7);

\coordinate (Vy) at (0,0.7,1.5);   
\coordinate (Vz) at (0,1.5,0.7);   

\draw[very thick, red] (V) -- (O);
\draw[very thick, red] (V) -- (Vy);
\draw[very thick, red] (V) -- (Vz);

\fill[red] (V) circle (1.5pt);

\draw[very thick, red] (O) -- (X);
\fill[red] (O) circle (1pt);

\end{tikzpicture}
$$

$$
\begin{tikzpicture}[scale=2.5]

\coordinate (O) at (0,0,0);

\coordinate (X) at (1.5,0,0);
\coordinate (Y) at (0,1.5,0);
\coordinate (Z) at (0,0,1.5);

\draw[thick,->] (O) -- (X) node[below right] {$x$};
\draw[thick,->] (O) -- (Y) node[above left] {$y$};
\draw[thick,->] (O) -- (Z) node[above right,xshift=-10pt,yshift=-10pt] {$z$};

\fill[gray!40,opacity=0.15] ($(O)+(-0.1,-0.1,-0.1)$) circle (0.15);

\shade[bottom color=blue!10, top color=blue!40, shading angle=-100, opacity=0.6]
  (O) -- (Y) -- ($(Y)+(0,0,1.5)$) -- (Z) -- cycle;

\shade[bottom color=red!10, top color=red!40, shading angle=0, opacity=0.6]
  (O) -- (X) -- ($(X)+(0,0,1.5)$) -- (Z) -- cycle;

\shade[bottom color=green!10, top color=green!40, shading angle=100, opacity=0.6]
  (O) -- (X) -- (Y) -- cycle;

\draw[thick] (O) -- (X) -- ($(X)+(0,0,1.5)$) -- (Z) -- cycle;
\draw[thick] (O) -- (Y) -- ($(Y)+(0,0,1.5)$) -- (Z) -- cycle;
\draw[thick] (O) -- (X) -- (Y) -- cycle;

\fill[black] (O) circle (1.5pt);
\node[below left,xshift=6pt,yshift=-3pt,black] at (O) {$0$};


\end{tikzpicture}
$$
\caption{An example of the tropical dimension filtration in three steps.}\label{fig: dim-filt}
\end{figure}
\end{example}

It is straightforward to see that truncation below dimension $r$ is partially balanced in the sense that it satisfies the balancing condition about all faces of dimension $r$ and above. 

If $\dim (\Gamma,k)$ is $m$, then by construction, we have maps of stratified spaces
\[
\Gamma = \Gamma_{\geq 0}\to \Gamma_{\geq 1}\to\cdots\to\Gamma_{\geq m}
\]
Each space refines the next. This is called the {\it tropical dimension filtration}. 

Intuitively, the space $\Gamma_{\geq m}$ should be thought of as a cone complex of pure dimension $m$. Moving left in the filtration by one steps introduces cells of increasing dimension, starting with dimension $m-1$. In the forward direction, each map ``erases'' faces when they can be erased. 

When we add a cell, these can either be embedded in cells of higher dimension, with some unconstrained integer decoration, or can they ``float'', i.e. they map to cells where the $K$-weight is $0$.

We now explain a situation in which balancing can be checked in a straightforward manner. Fix any cell $\gamma$ in $\Gamma_{\geq r}$ and consider the refinement
\[
\gamma'\to \gamma
\]
induced by $\Gamma_{\geq r-1}\to\Gamma_{\geq r}$. Consider an $r-2$-dimensional cell $\delta$ of $\gamma'$ along which $r-1$-dimensional cells $\tau_1,\ldots,\tau_s$ meet.

Let $k_\gamma$ be the $K$-weight of $\gamma$. The {\it excess-$\chi$} at $\tau_i$ is 
\[
k_{\tau_i}-k_\sigma. 
\]
Maintain the notation above for the following lemma. 

\begin{lemma}
Suppose all $r-1$-dimensional cells that contain $\delta$ in their closure are contained inside $\gamma$. Replacing the $K$-weight of $\tau_i$ with its excess-$\chi$, the traditional balancing condition is satisfied at $\delta$. 
\end{lemma}

\begin{proof}
Consider the subdivision $\Gamma_{\geq r-1}\to\Gamma_{\geq r}$, and equip $\Gamma_{\geq r-1}$, instead of its given $K$-weights, with its {\it pullback} $K$-weights. With these weights, balancing certainly holds at $\delta$. The lemma follows from the $K$-balancing formula. Indeed, since we have balancing with both the old weights and the new weights, it follows immediately that the excess-$\chi$ terms must cancel. 
\end{proof}


\subsubsection{$K$-balanced projections}\label{sec: K-projections} Let $(\Gamma,k)$ be a $K$-weighted polyhedral complex in $N_\RR$ of dimension $m$. A useful way of extracting information from the $K$-balancing condition is by linear projections. Consider a linear map
\[
b\colon N_{\mathbb R}\to L_{\mathbb R}
\]
induced by a subspace $L^\vee$ of the character space $M$. We will mostly be interested in this when $\dim L_{\mathbb R}$ is equal to $m$ or less -- when $|\Gamma|\to L_{\mathbb R}$ is plausibly surjective. 

Given $N_{\mathbb R}\to L_{\mathbb R}$ and $\Gamma$, possibly after dilating, replacing $\Gamma$ with a subdivision and equipping it with the pullback weight, we can find a polyhedral decomposition $\mathcal L$ of $L_{\mathbb R}$ such that the induced map
\[
\Gamma\to\mathcal L
\]
is {\it combinatorially flat} -- every polyhedron of the domain maps surjectively onto its image polyhedron. 

We now equip the maximal dimensional cells of $\mathcal L$ with a {\it pushforward weight}. We first explain this conceptually, and then record the required practical consequence. 

Fix a vertex $V$ of $\mathcal L$ and a full-dimensional cone $\sigma_V$ in the star of $V$. Let $V_1,\ldots, V_b$ be the vertices of $\Gamma$ mapping to $V$. 

The polyhedral decomposition $\mathcal L$ gives rise to a degeneration-compactification of the torus $L\otimes\mathbb G_m$. We denote it $B_{\mathcal L}$. There is a map
\[
b\colon Y_\Gamma\to B_{\mathcal L}.
\]
The vertex $V$ determines a choice of a component of $B_{\mathcal L}$. The choice of cone $\sigma$ determines a $0$-dimensional stratum on that component. Let $p_{V,\sigma}$ denote this point. Note also that $(\Gamma,k)$ determines a $K$-balanced integer decoration on the star of every vertex and therefore a $K$-theory class on $Y_\Gamma$. Denoted it by $\alpha(\Gamma,k)$

\begin{definition}
Given $\Gamma\to\mathcal L$ as above, the pushforward weight at $V$ along $\sigma$ is the Euler characteristic of the product of $b^{\star}([\mathcal O_{p_{V,\sigma}}])$ with the class $\alpha(\Gamma,k)$. 
\end{definition}

Note that $b$ is flat by construction, e.g. from the polyhedral criterion for flatness of toric morphisms~\cite[Section~4]{AbramovichKaru}, so the pullback $b^\star$ is well-defined. 

\begin{proposition}[Projections are $K$-balanced]\label{prop: pushforwards}
The pushforward weight of $\Gamma\to\mathcal L$ at a vertex $V$ along any maximal cone is independent of $V$ and choice of cone. Furthermore, this weight can be computed from the asymptotic weight of $\Gamma$. 
\end{proposition}

\begin{proof}
The structure sheaves of different points on $B_\mathcal L$ define the same class in $K$-theory, so the independence from the choice of cells follows immediately. By choosing an unbounded cell in $\mathcal L$ to do the computation, we see the computation can be done using the asymptotic weight. 
\end{proof}

Proposition \ref{prop: pushforwards} is closely related the notion of a {\it harmonic} morphism in tropical geometry. We conclude the section with some remarks that should clarify the statements.

\begin{remark}[The top dimensional cells]
In order to understand the pushforward weight, first consider the case where $\Gamma$ has dimension $m$ and $\mathcal L$ also has dimension $m$. In this case, the pushforward weight has a very simple interpretation. For each $m$-dimensional cell of $\mathcal L$, consider the $m$-dimensional cells that map surjectively onto it. Since $\Gamma$ has cells of dimension at most $m$, these cells also map {\it injectively}, but may do so via a linear map with non-unit determinant -- the {\it expansion factor}. In this case, the pushforward weight is the sum of the weights of these cells, where each is further multiplied by the absolute value of the determinant. 

The $K$-balancing condition is just the fact that these sums are the same. In other words, the map has a well-defined degree. 
\end{remark}

In this next remark, we explain in more detail how to practically compute the pushforward weight. One consequence of this will be to help us bound the possible slopes/directions of faces. 

\begin{remark}[Intermediate cells]\label{rem: intermediate cells}
Now consider $\Gamma\to\mathcal L$ where the dimension of $L$, which we denote $\ell$, is potentially {\it smaller} than $m = \dim \Gamma$. Fix a cell $\lambda$ in the target $\mathcal L$. There cells that map surjectively onto $\lambda$ have dimensions potentially ranging from $\ell$ to $m$. 

The faces of $\Gamma$ that map onto $\lambda$ can be organized into an abstract simplicial complex via the incidence poset. Specifically, the vertices correspond to $\ell$-dimensional cells mapping onto $\lambda$. The edges correspond to $\ell+1$-dimensional cells, with the incidence structure provided by the face inclusion of $\ell$-faces into $(\ell+1)$-faces. Similarly, $2$-simplices are $\ell+2$-dimensional faces, and so on. 

This simplicial complex has geometric meaning in the context of the map of broken toric varieties
\[
b\colon Y_\Gamma\to B_{\mathcal L}.
\]
The cell $\lambda$ corresponds to a closed point. The vertices of this simplicial complex correspond to the irreducible components in the preimage, the higher dimensional cells are intersections, and so on.

The weight on $\lambda$ is defined above by the Euler characteristic of $\alpha(\Gamma,k)$ restricted to this fiber and we record how to compute it. Spelling out the Euler characteristic formulae from Section~\ref{sec: basic-formulae}, there is a contribution for each cell in the simplicial complex above, and we take the alternating sum of contributions, with sign depending on the dimension. The weight $k$ determines a decoration on every cell. We would like to simply take the alternating sum of these contributions, however, this is only true if every $\ell$-dimensional cell of $\mathcal L$ that maps to $\lambda$ does so with unit determinant. Indeed, if the determinants are unit, then the preimage of the point $p_\lambda$ under $b$ is a reduced union of complete intersection strata in a toric component of $B_\Lambda$. 

If the determinants are not unit, then the inclusion/exclusion formula is not a complete intersection of reduced hypersurfaces, but of non-reduced hypersurfaces. In this case, we can use the discussion of multiplicities in Section~\ref{sec: basic-formulae} -- there are universal formulae that describe the contributions in terms of those recorded by $k$. 

In the next section, we will use the following observation. If we fix the decoration values of $k$ on cells of $\Gamma$ of dimension {\it strictly larger than $\ell$} but consider increasing the integer decoration on an $\ell$-dimensional cells $\lambda'$ that maps to $\lambda$, then by the multiplicity discussion in Section~\ref{sec: basic-formulae}, the change in the pushforward weight is equal to the change in weight times the absolute value of the determinant of $\lambda'\to\lambda$. 
\end{remark}

\section{Finiteness for $K$-tropicalizations}\label{sec: finiteness-K}

We prove that the set of $K$-tropicalizations that arise as tropicalizations of a transverse quotient sheaf with fixed Hilbert polynomial is finite. 

Fix a pair $(X,D)$ and a logarithmically flat sheaf $\mathcal E$. Consider logarithmic quotients on expansions $p\colon X_\Gamma\to X$, with fixed asymptotic $K$-weights, denoted $\Lambda$. Each logarithmic quotient $p^\star\mathcal E\twoheadrightarrow \mathcal F$ determines a $K$-tropicalization -- a piecewise linear space with $K$-weights. A $K$-tropicalization has a combinatorial type. A combinatorial type of a piecewise linear space is defined in~\cite{Logquot}. Here we demand that the $K$-weights are respected. In short, taking combinatorial type equates two $K$-tropicalizations if they sit in a piecewise linear family, compatible with the $K$-weights. 

\begin{theorem}\label{thm: combinatorial-boundedness}
The set of combinatorial types of $K$-tropicalizations of points of $\mathsf{Quot}_\Lambda(X|D,\mathcal E)$ is finite. 
\end{theorem}

\subsection{The general case from the toric case} We reduce the general finiteness statement for $K$-tropicalizations to the toric case. Recall that $X$ is a smooth projective variety with an snc divisor $D$ with components $D_1,\ldots, D_k$.

\begin{proposition}\label{prop: global-toric-embedding}
There is a strict embedding
\[
X\hookrightarrow Y
\]
where $Y$ is the projectivization of a direct sum of line bundles over a product of projective spaces $\mathbb P^{\underline n}$, where $Y$ is equipped with its fiberwise toric logarithmic structure, denoted $\partial Y$.
\end{proposition}

Note $Y$ is abstractly a toric variety, but we have only equipped it with a snc divisor coming from the projective bundle directions, which is a union of torus invariant divisors. The remaining toric divisors, not included in $\partial Y$, are a union of pullbacks of hyperplane sections from the factors of $\mathbb P^{\underline n}$. 

\begin{proof}
We have seen that we can embed $X$ into a projective bundle $\mathbb P$ over $X$, and this is the projectivization of a direct sum of line bundles on $X$. Let $L_1,\ldots, L_k$ be the line bundles. Since $X$ is projective, we can express each $L_i$ as a difference of hyperplane sections on $X$ and so as the pullback of a line bundle of the form $\mathcal O(1,-1)$. Using this embedding, the bundle $\mathbb P$ is pulled back from the projectivized direct sum of corresponding line bundles on this product of projective spaces. 
\end{proof}

We can now state the reduction step. 

\begin{proposition}\label{prop:Toric Case Enough}
Assume Theorem~\ref{thm: combinatorial-boundedness} for pairs $(Y,E)$ consisting of a toric variety and its toric boundary. Then Theorem~\ref{thm: combinatorial-boundedness} holds for all pairs $(X,D)$. 
\end{proposition}

\begin{proof}
Consider tropicalizations of points in the moduli space $\mathsf{Quot}_\Lambda(X|D,\mathcal E)$. We exhibit an injective map from the set of tropicalizations arising from $\mathsf{Quot}_\Lambda(X|D,\mathcal E)$ into a corresponding set of tropicalizations arising from a toric pair, with controlled asymptotics $K$-weights. 

We start with the geometric input from Proposition~\ref{prop: global-toric-embedding}. Let $Y$ be the toric variety arising in the proposition, and 
\[
X\hookrightarrow Y
\]
the strict embedding. Recall that $Y\to \mathbb P^{\underline n}$ is a projective space bundle and we have equipped $Y$ with a fiberwise toric logarithmic structure. 

Consider a point of $\mathsf{Quot}_\Lambda(X|D,\mathcal E)$. It is represented by an algebraically transverse sheaf on an expansion $X_\Gamma$ of $X$ along $D$. Since $X\hookrightarrow Y$ is strict we can express $X_\Gamma$ as a pullback of an expansion
\[
Y_\Gamma\to Y
\]
along $j\colon X\hookrightarrow Y$. A logarithmic quotient $\mathcal F$ of $\mathcal E$ on $X_\Gamma$ pushes forward to a logarithmic quotient $j_\star\mathcal F$ of $j_\star\mathcal E$ on $Y_\Gamma$. Consider the ample line bundle $\mathcal O(1,\ldots,1)$ on $\mathbb P^{\underline n}$. Using the projection formula in $K$-theory, the $K$-weights of this quotient with respect to this ample and all the line bundles making up the constituent line bundles of $Y\to \mathbb P^{\underline n}$ is determined by the $K$-weights of $\mathcal F$.  

We now enlarge the logarithmic structure on $Y$ from $\partial Y$ to make it toric. Fix an algebraically sheaf $\mathcal F$ on $Y_\Gamma\to Y$. Consider the base $\mathbb P^{\underline n}$ of the toric bundle. On a $\mathbb P^{n_i}$ factor of the product, choose $(n_i+1)$ generic hyperplanes. Each choice of these hyperplanes makes $\mathbb P^{n_i}$ isomorphic to a toric variety, with a different torus for each choice. 

Similarly, choosing such a set of generic hyperplanes in each factor, we turn $\mathbb P^{\underline n}$ into a toric variety, and by pulling back $Y\to \mathbb P^{\underline n}$, such a choice of hyperplanes makes $Y$ into a toric variety as well. We refer to this system of hyperplane sections as a {\it system of residual toric divisors}. Note that different choices of toric divisors give rise to isomorphic toric varieties, in fact, they are related by a factorwise projective change of coordinates. 

\noindent
{\bf Claim.} Let $\mathcal H$ be a generic system of residual toric divisors. The sheaf $\mathcal F$ on $Y_\Gamma$ is algebraically transverse with respect to the union of $\partial Y$ and $\mathcal H$. 

The claim follows from a Bertini argument -- given $\mathcal F$ algebraically transverse to $\partial Y$, a generic section of a basepoint free line bundle on $Y$ not contain any sections of $\mathcal F$, and the same is true on any stratum. Repeating this for all the elements forming a system of residual toric divisors proves the claim. We can do this for every for every logarithmic quotient. 

Putting this together, we obtain, for each quotient $\mathcal F$ on $X_\Gamma$, a quotient on an expansion $Y_\Gamma$ of a toric variety, along its full toric boundary, with predictable asymptotic $K$-weights depending only on the starting from the asymptotic $K$-weights of $\mathcal F$. The tropicalization of the sheaf with respect to this larger logarithmic structure $E$ certainly determines the one respect $\partial Y$. We emphasize here that that the choice of residual divisor system $\mathcal H$ will depend on $\mathcal F$, by changing coordinates as above, we can view each of the resulting quotient as an object on a fixed toric variety. 

We have therefore constructed an injective map from the set of all tropicalizations arising from the $(X,D)$-problem into the set of all tropicalizations arising in the $(Y,E)$ problem, and the latter is toric, so the proof of the proposition is complete. 
\end{proof}

\subsection{Finiteness from projections} We continue in the toric context, with cocharacter vector space $N_{\mathbb R}$. 

Let $\Gamma_{\geq r}$ be a polyhedral complex in $N_{\mathbb R}$ with a $\mathbb Z$-weighting on all faces of codimension up to $r$. We require the $K$-balancing condition to hold on the star of all faces of codimension up to $r-1$. We are slightly abusing notation, since we have defined  $\Gamma_{\geq r}$ without reference to a particular $\Gamma$. 

In this section, we will use the projections in Section~\ref{sec: K-projections} to study the following problem -- {\it if we fix the asymptotic $K$-weighted polyhedral complex, what can be said about the faces of possible
\[
\Gamma_{\geq r-1}\to \Gamma_{\geq r}
\]
that are compatible with these data?}

In other words, what are the ways in which the $r-1$-dimensional cells can be added if fix the cells of dimension $r$ and above and the asymptotics are both fixed. 

Furthermore, in this section, we will always assume that $\Gamma_{\geq r-1}$ is minimal in the sense that it is equal to its own canonical coarsening. We start with the following construction. 

\begin{construction}[Slicing $K$-tropicalizations]\label{con: slicing}
Let $\Gamma$ be a $K$-balanced polyhedral complex in $N_{\mathbb R}$. Consider an integral linear projection
\[
\beta \colon N\to L
\]
to a quotient lattice of rank $\ell$. Replace $\Gamma$ with a dilation-subdivision and let $\mathcal L$ be a decomposition of $N_{\mathbb R}$ such that
\[
\Gamma\to \mathcal L
\]
is combinatorially flat. Then for any point $p$ in the interior of an $\ell$-dimensional cell of $\mathcal L$, the fiber of $p$ is a $K$-balanced complex in the latticed vector space\footnote{Strictly speaking, this is a torsor for a latticed vector space. Choose any origin to turn it into a vector space. } $\beta^{-1}(p)$. The $K$-weight on each cell is given by Remark~\ref{rem: intermediate cells}. Precisely, observe that there is a natural bijection between cells of the fiber and the cells of $\Gamma$ that map surjectively onto the cell containing $p$. The cells have been given a weight by the procedure of Remark~\ref{rem: intermediate cells}, and we take the induced one on the fiber.  

Furthermore, we have the following:

\noindent
{\bf Claim}. The asymptotic $K$-weight of this $K$-weighted fiber depends only on the asymptotic $K$-weight of $\Gamma$ and the projection $\beta$. 

\noindent
To see this, first consider the decoration on the unique $0$-dimensional cone of the asymptotic weight of this fiber that we have just described. We can view
\[
\Gamma\to \mathcal L
\]
as defining a degeneration-compactification of the map 
\[
(N\to L)\otimes \mathbb G_m
\]
over $\mathbb A^1$, with the general fiber given by the map of toric varieties defined by recession fans of $\Gamma$ and $\mathcal L$. The formula of Definition~\ref{def: asymptotic-weights} defines a $K$-weight on the general fiber of the domain that specializes to the $K$-theory class defined by $\Gamma$. Furthermore, this general fiber $K$-theory class is defined by the asymptotic $K$-weight, essentially by construction. Similarly, the point on the target defined by the choice of cone is the specialization of the class of a point in the general fiber. By formal properties of $K$-theory classes under specialization, we see that the asymptotic decoration on the $0$-cone is computed from the asymptotic $K$-weight of $\Gamma$ by intersecting with a fiber of this map between general fibers, and therefore can be computed as claimed. 
\end{construction}

We use Construction \ref{con: slicing} in inductive arguments about boundedness. We begin with the following, essentially as a warmup. 

Fix a toric pair $(Y,E)$ with cocharacter lattice $N$, and fix a $K$-weight $k_0$ on the fan $\Sigma_Y$. Let $\Gamma_{\geq r}$ be a truncation below dimension $r$ of a $K$-weighted polyhedral complex -- as noted above, this is a complex with decoration on cones of dimension $r$ and above, satisfying the $K$-balancing at stars of all faces of dimension $r-1$ and above. 

Let $\mathbb F(k_0,\Gamma_{\geq r})$ denote the set of all $K$-weighted polyhedral complexes in $N_{\mathbb R}$ that arise from algebraically transverse quotient sheaves on expansions of $Y$, whose asymptotic $K$-weight is $k_0$, and whose truncation above dimension $r$ is $\Gamma_{\geq r}$. 

\begin{proposition}\label{prop:PluckerBound}
Assume finiteness of $K$-tropicalizations for cocharacter spaces of dimension strictly less than $\dim N_{\mathbb R}$. 

Fix a vertex of $\Gamma_{\geq r}$. The set of linear spans of $r-1$-dimensional faces of elements of $\mathbb F(k_0,\Gamma_{\geq r})$ that meet $V$ is finite. 
\end{proposition}

\begin{proof}
Consider a face $\tau$ of $\Gamma$ of dimension $r-1$. We will bound the Pl\"ucker coordinates of $\tau^{\sf span}$. To do this it it suffices to show that, for each coordinate subspace projection
\[
\pi\colon N_{\mathbb R}\to L_{\mathbb R},
\] 
the determinant of the induced map $\tau^{\sf span}\to L_{\mathbb R}$ is bounded in absolute value. Let $\Gamma$ be a $K$-weighted polyhedral complex in $\mathbb F(k_0,\Gamma_{\geq r})$. Consider the set of vertices of $\Gamma$ that (i) map to $\pi(V)$ under the map above, and (ii) have an $(r-1)$-dimensional face incident to them that maps with nonzero determinant to $L_{\mathbb R}$. We claim that there is an a priori bound on the number of such vertices in terms of the asymptotic weights. In fact, there is such an a priori bound on the number of faces of any dimension mapping to a given face. 

To see this, observe first that by the $K$-balancing condition implies that, if the $\pi$-image of an $(r-1)$-dimensional cell incident to a point $V'$ in $\Gamma$ has dimension $(r-1)$, then $\pi$ is locally surjective at $V'$. Indeed, this applies from applying the $K$-balancing condition for projections in Section~\ref{sec: K-projections} at the star of $V'$. Now choose a flattening of the map $|\Gamma|\to L_{\mathbb R}$
\[
\Gamma\to\mathcal L
\]
and consider any $(r-1)$-dimensional containing $\pi(V)$. Now slice $\Gamma$ by taking the preimage of a generic point in the interior of this cell. Each vertex satisfying (i) and (ii) above contributes a face in the canonical coarsening of of this fiber. By the discussion in Construction~\ref{con: slicing}, and by the assumption that we know the result in all smaller dimensional ambient spaces, we have an a priori bound on the number of such faces as claimed.

Having bounded the number of such vertices, consider the projection $\Gamma\to\mathcal L$. The truncation below dimension $r$ has been fixed, we have fixed the Euler characteristics of intersections of potential algebraically transverse sheaves with strata of codimension $r$ and higher. We therefore obtain a lower bound on the possible Euler characteristic. 

Now observe that the pushforward weight on the cells of $\mathcal L$ is known from the asymptotic data and the possible integer decorations are bounded below. By the multiplicity discussion in Section~\ref{sec: basic-formulae}, the pushforward weight is, to first order, the absolute value of determinant value on the cell times the weight, up to higher codimension correction terms. In particular, these higher codimension terms are fixed by fixing $\Gamma_{\geq r}$. We therefore obtain an upper bound on the $K$-weight on any $(r-1)$-cell with nonzero determinant, and whose integer $K$-weight is larger than the minimum value. The minimum value can only be obtained as a pullback \cite[Step 2 of proof Lemma 6.3.2]{Logquot}, and so cannot appear in the canonical coarsening. We conclude the required bound. 
\end{proof}

\begin{corollary}\label{corr:JustBddNoStrata}
    Fix a positive integer $s$. There are finitely many combinatorial types of $K$-tropicalization with at most $s$ strata.
\end{corollary}

\begin{proof}
    We will bound the number of pairs consisting of a combinatorial type of $K$-tropicalization, and an order on its vertices. An order on a set of $s$ vertices is equivalently a bijection between vertices and the set $[s] = \{1,...,s\}$. Since we can order the vertices in any combinatorial type, this will complete our proof.

    Observe a $K$-tropicalization is determined by the data
    \begin{enumerate}[(i)]
        \item The location of each of its $s$ vertices in $N_S$.
        \item Discrete data which records the following information. First collections of vertices whose convex hull contain local cones. Second, for each local cone $\kappa$ the data of $\langle \kappa - \kappa \rangle $. By Proposition \ref{prop:PluckerBound} there are finitely many possibilities for the latter. Clearly there are finitely many possibilties for the former. We write $D$ for this finite set.
    \end{enumerate}
    There is thus an injection $$\text{Tropicalizations with $s$ vertices}\hookrightarrow N_X^s\times D.$$
    By \cite[Section 2.2.3]{Logquot}, for fixed $d\in D$ there are finitely many types in the preimage of $N_X^s\times \{d\}$.
\end{proof}

We record a generalization of the result above. We need some terminology first. Given a piecewise linear stratification $\Gamma$ of $N_{\mathbb R}$ and a locally closed stratum $\kappa$ of it, define 
\[
B_\epsilon^\circ(\kappa) = \{p\in N_{\mathbb R}\colon \textnormal{there exists } q\in\kappa \ \ s.t. \ \ |p-q|<\epsilon\}. 
\]
Intuitively, this is a ``tubular neighborhood'' of the stratum. It is, in fact, the star of the stratum in the sense of stratified spaces, but we have spelled it out for metric-minded readers. 

\begin{definition}
Let $\overline \Gamma\to \Gamma$ be a subdivision of piecewise linear stratifications of $N_{\mathbb R}$. The {\it incidence to $\kappa$} is the equivalence class of combinatorial types of stratified spaces
\[
\overline\Gamma\cap B_\epsilon^\circ(\kappa)
\]
over all $\epsilon>0$, where the equivalence relation identifies $\overline\Gamma\cap B_\epsilon^\circ(\kappa)$ and $\overline\Gamma\cap B_{\epsilon'}^\circ(\kappa)$ whenever they coincide under restriction. 
\end{definition}

A face $\kappa$ is said to be incidence to a face $\tau$ if the intersection of $\kappa$ with any sufficiently small ball around that $\tau$ is an element of the incidence. We now state our more refined finiteness statement. 

\begin{proposition}\label{prop: directions-bounded-at-faces}
Assume finiteness of $K$-tropicalizations for cocharacter spaces of dimension strictly less than $\dim N_{\mathbb R}$. Fix a face $F$ of $\Gamma_{\geq r}$ of dimension at most $r-1$. The set of linear spans of $r-1$-dimensional faces of elements of $\mathbb F(k_0,\Gamma_{\geq r})$ that are incident to $F$ is finite. 
\end{proposition}

\begin{proof}
We proceed similarly to the statement above, and bound the Pl\"ucker coordinates of the linear spans of faces incident to $F$. We can do so by projecting onto the $(r-1)$-dimensional coordinate spaces of $N_{\mathbb R}$ and bounding the determinants on the faces. 

Consider an $r$-dimensional cell incident to $F$ and consider a coordinate projection
\[
N_{\mathbb R}\to L_{\mathbb R}
\]
and flatten to a map of polyhedral complexes $\Gamma\to\mathcal L$. There is again an a priori bound, based on the asymptotics, for the number of cells of dimension $(r-1)$ mapping to any given cell of dimension $(r-1)$ in $\mathcal L$. Indeed, these faces correspond to faces in the slice obtained by taking the preimage of a point in the interior of this cell. And these are bounded by the induction hypothesis on the dimension. 

As before, there is a lower bound on the Euler characteristic of any $(r-1)$-cell, since we have fixed all weights on cells of dimension $r$ and larger, and by the same argument as in the previous proposition, we obtain a bound on the determinants of any such face with non-minimal Euler characteristic, and therefore any face appearing in the canonical coarsening. 
\end{proof}

Similarly, we can bound the number of cones incident to a given face. 

\begin{proposition}\label{prop: incidence is bounded}
Assume finiteness of $K$-tropicalizations for cocharacter spaces of dimension strictly less than $\dim N_{\mathbb R}$. Fix a face $F$ of $\Gamma_{\geq r}$ of dimension at most $r-1$. There is a finite upper bound for the number of $r-1$-dimensional faces of an element of $\mathbb F(k_0,\Gamma_{\geq r})$ incident to $F$. 
\end{proposition}

\begin{proof}
By Proposition~\ref{prop: directions-bounded-at-faces}, the possible linear spans are finite. So it suffices to show that number of faces incident to $F$ with a given linear span is bounded above. Fix such a linear span. We can again see the finiteness using projections. Consider the quotient
\[
N_{\mathbb R}\to L_{\mathbb R}
\]
by the linear span of $F$ itself. The intersection of all faces with this linear span with the given face $F$ maps to the origin in the quotient. Now project further to $L_{\mathbb R}'$, the image of the fixed linear span above. Now flatten the map $|\Gamma|\to L'_{\mathbb R}$ Since there is an a priori bound on the number of faces of any dimension in a preimage of a point in $L'_{\mathbb R}$, the number such faces is also bounded, as required. 
\end{proof}

\subsection{Balanced cycles} 
A \textit{tropical cycle} of dimension $r$ is a $K$-tropicalization $\Gamma$ such that $\Gamma = \Gamma_{\geq r}$, and which has weight zero on all strata of dimension larger than $r$. A tropical hypersurface is a tropical cycle of codimension one.

We will first bound the number of possibilities for the type of a balanced tropical hypersurface of dimension $n-1$ inside $N_X = \mathbb{Z}^n$. We fix asymptotic data $(\Gamma_\infty,k_\infty)$.

\begin{proposition}\label{prop:bddTropHyp}
    There are finitely many combinatorial types of tropical hypersurface of dimension $n-1$ with asymptotics $(\Gamma_\infty,k_\infty)$.
\end{proposition}
\begin{proof}
    Given a fan which is also a tropical cycle $(\Gamma_\infty,k_\infty)$ we write $P_{\Gamma_\infty}$ for its dual polytope. Fulton and Sturmfels show that the map $$\Gamma_\infty \rightarrow P_{\Gamma_\infty}$$ defines a bijection between the set of hypersurfaces with a single vertex at the origin, and the set of polytopes \cite[Theorem 4.2]{IntToricVar}. 
    The polytope associated to each vertex of $\Gamma$ tesselate to define a subdivision of $P_{\Gamma_\infty}$, thus defining a map $$D\colon \text{Tropical hypersurfaces } \rightarrow \text{Subdivisions of } P_{\Gamma_\infty}.$$
    Two tropical hypersurfaces $\Gamma,\Gamma'$ have the same image under $D$ if and only if they have the same combinatorial type.
    But the codomain of $D$ is clearly finite. Indeed each face in a choice of subdivision $P_{\Gamma_\infty}$ is a set of integral points in $P_{\Gamma_\infty}$, and a subdivision is characterised by its faces.
\end{proof}

To control the number of types of tropical cycle of higher codimension, consider a homomorphism $$\pi\colon \mathbb{Z}^n\rightarrow\mathbb{Z}^{k+1}.$$ Define the \textit{pushforward} of a tropical cycle $\Gamma$ of dimension $k$ in $\mathbb{Z}^{k+1}$ to be the tropical cycle $(\Gamma_\pi,k_\pi)$ defined as follows. First, define $\Gamma_\pi$  to be any smooth polyhedral decomposition of $\mathbb Z^{k+1}\otimes\mathbb R$ such that the map
\[
\Gamma\to\Gamma_\pi
\]
is combinatorially flat -- every cell maps surjectively onto a cell. The weight function $k_\pi$ is defined on a $k$-dimensional cell $\sigma$ of $\Gamma_\pi$ is a sum of multiples of weights of $k$-dimensional cells of $\Gamma$ mapping onto $\sigma$. Precisely, if 
\[
\widetilde \sigma_1,\ldots,\widetilde\sigma_t\mapsto \sigma,
\]
then each $\widetilde\sigma_i$ contributes  the weight $k(\widetilde \sigma_i)$ times the absolute value of the determinant of the linear map
\[
\mathsf{Span}(\widetilde \sigma_i)\to\mathsf{Span}(\sigma).
\]

\begin{lemma}
    For any group homomorphism $\pi$ the pushforward cycle is balanced.
\end{lemma}
\begin{proof}
The result is in fact a special case of Proposition~\ref{prop: pushforwards}. To say a word more, we note the balancing condition can be checked in the star of each $(k-1)$-dimensional face, so we immediately reduce to the case where $\Gamma$ is a fan. We are then simply computing the definition of the pushforward of Chow homology classes according to~\cite{IntToricVar}. Balancing is a formal consequence. 
\end{proof}

Proposition \ref{prop:PluckerBound} implies that there exists a finite set $W$ of elements in $N_X^\vee = M_X$ such that every stratum in $\Gamma$ is the intersect of sets of the form $$\{x|A(x)-b>0\}$$ for some $A\in W$ and $b\in \mathbb{R}$. We take $W$ a minimal set of linear functions with this property and which also contains all coordinate projections.

Write $S$ for the set of projections to $\mathbb{Z}^{k+1}$ defined by $k+1$ linearly independent elements of $W$. The cardinality of $S$ is at least $\binom{n}{k+1}$ and no more than $\binom{|W|}{k+1}$. 

\begin{lemma}\label{lemma:DeterminedByProjections}
    Let ${\kappa} \subset N_X$ be a stratum of $\Gamma$ of dimension $s$. There is an equality $${\kappa}=\bigcap_{F\in S} F^{-1}(F(\kappa))$$ 
\end{lemma}

\begin{proof}
    Elementary set theory ensures that $${\kappa}\subset \bigcap_{F\in S} F^{-1}(F({\kappa})).$$
    For the converse suppose $y\not\in \kappa$ so there is some $A\in W$, $b \in N_X$ such that $y \not\in \{x|A(x)-b> 0\}$ but ${\kappa}\subset \{x|A(x)-b> 0\}$. Choose any $s$ elements $A_1,...,A_{s}$ of $W$ such that $A,A_1,...A_{s}$ are linearly independent, which we may do because $W$ includes coordinate projections. The tuple $(A,A_1,...,A_{s})$ defines an element of $S$, and thus a projection to $\mathbb{Z}^{s+1}$. By construction $y$ lies in a different stratum to the image of $\kappa$ under this element of $S$: indeed the image of $\kappa$ is necessarily a union of strata, and the image of $y$ cannot lie in the image of $\kappa$. 
\end{proof}

\begin{proposition}\label{prop: TropCyclesBdd}
    There are finitely many combinatorial types of tropical cycle with asymptotics $(\Gamma_\infty, k_\infty)$
\end{proposition}
\begin{proof}
    Suppose $\Gamma_\infty$ has dimension $s<n$. If $s=n-1$ then the result follows from Proposition \ref{prop:bddTropHyp}.
    Otherwise consider the set of pushforwards of $\Gamma$ along elements of $S$. Write $(\Gamma_F,k_F)$ for the pushforward cycle arising from the projection along $F$ of a cycle $(\Gamma,k)$. 
    
    The asymptotics of $(\Gamma_F,k_F)$ are determined by $(\Gamma_\infty, k_\infty)$ and the map $F$. Thus by Proposition \ref{prop:bddTropHyp} there are finitely many possibilities for the combinatorial type of $(\Gamma_F,k_F)$. It therefore suffices to show that, after fixing the type of every $(\Gamma_F,k_F)$, there are finitely many possible types for $(\Gamma,k)$.

    By Corollary \ref{corr:JustBddNoStrata} it suffices to bound the number of strata appearing in $(\Gamma,k)$.  Write $\mathcal{P}$ for the common refinement over all $F$ of the intersect of the preimage of strata in $\Gamma_F$. The proof is completed if we can show that every stratum in $\Gamma_F$ is a union of strata in $\mathcal{P}$, but this holds by Lemma \ref{lemma:DeterminedByProjections}.
\end{proof}
\subsection{Global boundedness}

\begin{proof}[Proof of Theorem \ref{thm: combinatorial-boundedness}]
    By Proposition \ref{prop:Toric Case Enough} it suffices to handle the toric case. We will prove the following statement for every $n$ and $r$:

\noindent
\textit{There are finitely many possibilities for the type of $(\Gamma_r, k_r)$, whenever $\Gamma$ is the $K$-tropicalization of a point in $\mathsf{Quot}_\Lambda(X \mid D, \mathcal{E})$.}

We first handle the base cases. Where $n=0$ the result is vacuous. When $r=n$ the only possible $(\Gamma,k)$ is a single stratum $\Gamma$; in this situation $k$ adopts a value determined by the asymptotics.

    For the inductive step, suppose our claim is known for some fixed $R$ in the range $N
\geq r\geq R\geq 1$ with $n=N$ and also for all $n< N$ and with any $r$. There is a subdivision $$\varpi\colon \Gamma_{R-1} \rightarrow \Gamma_{R}$$ and we are required to show that there are finitely possible types for $\Gamma_{R-1}$. By Corollary \ref{corr:JustBddNoStrata} it suffices to bound the number of strata in $\Gamma_{R-1}$. We will establish the case $r=0$ at the end, and for now assume $R\geq 2$. There are finitely many possible types for $\Gamma_{R}$ by the inductive hypothesis, and each type has finitely many strata. Thus, it will suffice to fix one type $\Gamma'_R$, one stratum $\kappa$ in $\Gamma'_R$, and bound how many strata of $\Gamma_{R-1}$ have image in $\kappa$.

    Suppose $\kappa$ has dimension $y\geq R$. Choose a map $\pi\colon N_X\rightarrow \mathbb{Z}^y = N_\kappa$ which induces an isomorphism of lattices $\langle\kappa-\kappa\rangle^\mathsf{gp}\rightarrow \mathbb{Z}^{y}$. We construct a balanced tropical cycle $(\Gamma_\kappa,k_\kappa)$ in $N_\kappa$ of pure dimension $R-1$.  

    To construct the stratification $\Gamma_\kappa$ we aim to take the pushforward along $\pi$ of $\Gamma_R\cap B_\epsilon(\kappa)$, and equip it with the $\kappa$-excess weight function. 
    A subtlety arises: the image of $\Gamma_R\cap B_\epsilon(\kappa)$ is contained within an open subset of $N_\kappa$. 
    To overcome this issue, extend in the obvious way any stratum $\tau'=\pi(\tau)$ such that the image of $\pi$ does not contain limit points of $\tau$. Where the extensions of strata 
    $\tau'_1,...\tau'_s$ intersect,
    we define strata of $\Gamma_\kappa$ to be the common refinement. Given a stratum $\gamma$ so formed, write $W_\gamma$ for the set of $\tau_i'$ such that $\gamma$ lies in the extension of $\tau'_i$. For a stratum $\gamma$ which arises from this construction, $k$ adopts the value $\sum_{\tau'\in {W_\gamma}} k^{\mathsf{excess}}(\tau')$, where the sum is over an extension $\tau$ containing $\gamma$.

    The asymptotics of $\Gamma_\kappa$ are determined by three data:
    \begin{enumerate}[(i)]
        \item The asymptotics of $\Gamma_R$, in particular cones which are unbounded in the directions in which $\kappa$ is unbounded.
        \item The incidence to $\kappa$, for which there are finitely many possibilities by Proposition \ref{prop: incidence is bounded}.
        \item The incidence to $\tau$ for each $\tau$ a face of $\kappa$. Once again there are finitely many possibilities by Proposition \ref{prop: incidence is bounded}.
    \end{enumerate}
    We deduce that there are finitely many possible combinatorial types for the asymptotics of $\Gamma_\kappa$. It follows from Proposition \ref{prop: TropCyclesBdd} that there are finitely many possibilities for the combinatorial type of $\Gamma_\kappa$. Since the number of strata in $\Gamma_\kappa$ is at least the number of strata $\tau$ in $\Gamma_R$ such that $\varpi(\tau) \subset \kappa$.

    Finally we handle the case $r=0$. The total Euler characteristic $\chi$ of $(\Gamma_\infty,k_\infty)$ is the inclusion-exclusion sum over all cones $\kappa$ of the weight $k_\infty(\kappa)$. To bound the number of strata in $\Gamma_0$ we need only bound the number of vertices added in the subdivision $$\Gamma_0\rightarrow \Gamma_1.$$ The total Euler characteristic, $\chi$ is the Euler characteristic of $\Gamma_2$ plus the sum of the excess Euler characteristics of these additional vertices. But each vertex we add contributes a positive integer to the excess Euler characteristic, so we obtain a bound for the number of vertices to be added.
\end{proof}

\section{Assembling the pieces}

Fix $(X,D)$, an algebraically flat coherent sheaf $\mathcal E$, and asymptotic $K$-weights $\Lambda$. To clarify the latter we fix, for every stratum of $(X,D)$ including $X$ itself, a Hilbert polynomial. Fix also an ample on $X$.

A geometric point of the logarithmic Quot space $\mathsf{Quot}(X|D,\mathcal E)$ can be represented by a quotient sheaf $\mathcal F$ of the pullback of $\mathcal E$ on an expansion
\[
X_\Gamma\to X.
\]
For each closed stratum of $X_\Gamma$ we obtain a Hilbert polynomial by restricting $\mathcal F$ to this stratum. According to Definition~\ref{def: asymptotic-weights}, we specialize these to obtain an asymptotic polynomial for each stratum of $(X,D)$. 

Let $\mathsf{Quot}_\Lambda(X|D,\mathcal E)$ be the subfunctor of $\mathsf{Quot}(X|D,\mathcal E)$ consisting of families that, fiberwise, can be represented by sheaves on expansions whose asymptotic decorations are given by $\Lambda$. We state a precise version of Theorem~\ref{thm: boundedness-quot-scheme}.

\begin{theorem}
The moduli space $\mathsf{Quot}_\Lambda(X|D,\mathcal E)$ is proper. 
\end{theorem}

The proof is now essentially formal -- one result says that there are finitely many expansions coming from Gr\"obner stratifications, and the other says that there is a finite dimensional Quot scheme parameterizing quotients on the components of any fixed expansion. We provide the details.

\begin{proof}
The asymptotic decorations are defined so they are locally constant in flat families, so the inclusion map for the subfunctor above satisfies the valuative criterion. The larger logarithmic Quot space $\mathsf{Quot}(X|D,\mathcal E)$ is shown to be separated and universally closed, so it remains to show that it is of finite type. 

To prove this, note that for each point of $\mathsf{Quot}_\Lambda(X|D,\mathcal E)$ we have an associated $K$-balanced weighted polynomial complex in $\Sigma(X,D)$, with polynomial decorations. For fix $\Lambda$, it follows from Theorem~\ref{thm: combinatorial-boundedness} that there are only finitely many combinatorial types of $K$-tropicalization that can arise. 

Fix a combinatorial type of $K$-tropicalization. The space of $K$-tropicalizations may only be a piecewise linear space rather than a cone complex, but by tropical boundedness, it admits a subdividion by a finite dimensional cone complex, and we can choose one arbitrarily. The choice of combinatorial type defines a family piecewise linear spaces, and we choose arbitrarily a further subdivision to make this representable. 

We have fixed a $K$-tropicalization and therefore an expansion. We now appeal to Theorem~\ref{thm: canonical-transversalization}. The logarithmic space $\mathsf{Quot}_\Lambda(X|D,\mathcal E)$ has a map to the set of $K$-tropicalizations. Let us record exactly how this works. 

A logarithmic quotient is represented by an actual quotient on an expansion $\widetilde X\to X$. The expansion determines a $K$-tropicalization by passing to the Gr\"obner stratification/tropical support. The boundedness question arises because in principle, it is not clear that a logarithmic quotient can be represented scheme theoretically on {\it any} subdivision of the Gr\"obner stratification. 

But applying Theorem~\ref{thm: canonical-transversalization}, given any representation of a logarithmic quotient, we can find an equivalent quotient -- equivalent with respect to the relation in Definition~\ref{def: quot-equivalence} -- on any expansion on a subdivision of $\Gamma$ such that the expansion is normal crossings. Choose such a subdivision. This determines an expansion
\[
X_\Gamma\to X
\]
whose scheme structure depends only on the combinatorial type of the $K$-tropicalization. The locus $\mathsf{Quot}_\Gamma(X|D,\mathcal E)$ that can be represented by algebraically transverse subschemes this fixed expansion $X_\Gamma$ is of finite type. Indeed, by the results of Theorem~\ref{thm: combinatorial-boundedness} the Hilbert polynomial of any quotient sheaf on any component of $X_\Gamma$ has finitely many choices, and so $\mathsf{Quot}_\Gamma(X|D,\mathcal E)$ is a locally closed subscheme in a Quot scheme with fixed Hilbert polynomial. Boundedness follows. 
\end{proof}
 
\bibliographystyle{siam}
\bibliography{LogHilbBoundedness}

\begin{thebibliography}{10}

\bibitem{ACMUW}
{\sc D.~Abramovich, Q.~Chen, S.~Marcus, M.~Ulirsch, and J.~Wise}, {\em
  Skeletons and fans of logarithmic structures}, in Nonarchimedean and tropical
  geometry, Simons Symp., Springer, [Cham], 2016, pp.~287--336.

\bibitem{ACMW17}
{\sc D.~Abramovich, Q.~Chen, S.~Marcus, and J.~Wise}, {\em Boundedness of the
  space of stable logarithmic maps}, Journal of the European Mathematical
  Society, 19 (2017), pp.~2783--2809.

\bibitem{AbramovichKaru}
{\sc D.~Abramovich and K.~Karu}, {\em Weak semistable reduction in
  characteristic 0}, Invent. Math., 139 (2000), pp.~241--273.

\bibitem{AP15}
{\sc D.~Anderson and S.~Payne}, {\em Operational {$K$}-theory}, Doc. Math., 20
  (2015), pp.~357--399.

\bibitem{BGS11}
{\sc J.~I. Burgos~Gil and M.~Sombra}, {\em When do the recession cones of a
  polyhedral complex form a fan?}, Disc. Comp. Geom., 46 (2011), pp.~789--798.

\bibitem{CMN}
{\sc F.~Carocci, L.~Monin, and N.~Nabijou}, {\em Chow theory of toric variety
  bundles}, arXiv:2411.16883,  (2024).

\bibitem{Cartwright}
{\sc D.~Cartwright}, {\em The {G}r{\"o}bner stratification of a tropical
  variety}, 2012.

\bibitem{ModStckTropCurve}
{\sc R.~Cavalieri, M.~Chan, M.~Ulirsch, and J.~Wise}, {\em A moduli stack of
  tropical curves}, Forum Math. Sigma, 8 (2020).

\bibitem{Dod24}
{\sc E.~Dodwell}, {\em Tropical geometry in torus bundles}, University of
  Cambridge, MPhil Thesis,  (2024).

\bibitem{IntToricVar}
{\sc W.~Fulton and B.~Sturmfels}, {\em Intersection theory on toric varieties},
  Topology, 36 (1997), pp.~335--353.

\bibitem{GKZ}
{\sc I.~Gelfand, M.~Kapranov, and A.~Zelevinksi}, {\em Discriminants,
  Resultants, and Multidimensional Determinants}, Modern Birkhauser Classics,
  1994.

\bibitem{GG16}
{\sc J.~Giansiracusa and N.~Giansiracusa}, {\em Equations of tropical
  varieties}, Duke Math. J., 165 (2016), pp.~3379--3433.

\bibitem{gross2012logarithmic}
{\sc M.~Gross and B.~Siebert}, {\em Logarithmic {G}romov-{W}itten invariants},
  J. Amer. Math. Soc., 26 (2013), pp.~451--510.

\bibitem{Gub:guideToTrop}
{\sc W.~Gubler}, {\em A guide to tropicalizations}, in Algebraic and
  combinatorial aspects of tropical geometry, vol.~589 of Contemp. Math., Amer.
  Math. Soc., Providence, RI, 2013, pp.~125--189.

\bibitem{GeometryModuliSheaves}
{\sc D.~Huybrechts and M.~Lehn}, {\em The geometry of moduli spaces of
  sheaves}, Cambridge Mathematical Library, Cambridge University Press,
  Cambridge, second~ed., 2010.

\bibitem{KSZ92}
{\sc M.~Kapranov, B.~Sturmfels, and A.~Zelevinsky}, {\em Chow polytopes and
  general resultants}, Duke Math. J., 67 (1992), pp.~189--218.

\bibitem{KP08}
{\sc E.~Katz and S.~Payne}, {\em Piecewise polynomials, {M}inkowski weights,
  and localization on toric varieties}, Algebra Number Theory, 2 (2008),
  pp.~135--155.

\bibitem{KKMSD}
{\sc G.~Kempf, F.~Knudsen, D.~Mumford, and B.~Saint-Donat}, {\em {Toroidal
  embeddings I}}, vol.~339, Springer, 2006.

\bibitem{Logquot}
{\sc P.~Kennedy-Hunt}, {\em The logarithmic {Q}uot space and its
  tropicalisation}, 2023.

\bibitem{LiWu}
{\sc J.~Li and B.~Wu}, {\em Good degeneration of {Q}uot-schemes and coherent
  systems}, Comm. Anal. Geom., 23 (2015), pp.~841--921.

\bibitem{MaclaganSturmfels}
{\sc D.~Maclagan and B.~Sturmfels}, {\em Introduction to Tropical Geometry},
  American Mathematical Society, 2015.

\bibitem{MR20}
{\sc D.~Maulik and D.~Ranganathan}, {\em Logarithmic {D}onaldson-{T}homas
  theory}, Forum Math. Pi, 12 (2024), pp.~Paper No. e9, 63.

\bibitem{MR23}
\leavevmode\vrule height 2pt depth -1.6pt width 23pt, {\em Logarithmic
  enumerative geometry for curves and sheaves}, Camb. J. Math., 13 (2025),
  pp.~51--172.

\bibitem{Mi03}
{\sc G.~Mikhalkin}, {\em Enumerative tropical algebraic geometry in
  $\mathbb{R}^2$}, J. Amer. Math. Soc., 18 (2003), pp.~313--377.

\bibitem{MolchoWise}
{\sc S.~Molcho and J.~Wise}, {\em The logarithmic {P}icard group and its
  tropicalization}, Compos. Math., 158 (2022), pp.~1477--1562.

\bibitem{MW23}
\leavevmode\vrule height 2pt depth -1.6pt width 23pt, {\em Remarks on
  logarithmic {\'{e}}tale sheafification}, arXiv:2311.05172,  (2023).

\bibitem{Mor93}
{\sc R.~Morelli}, {\em The {$K$}-theory of a toric variety}, Adv. Math., 100
  (1993), pp.~154--182.

\bibitem{Nak17}
{\sc C.~Nakayama}, {\em Logarithmic \'etale cohomology, {II}}, Adv. Math., 314
  (2017), pp.~663--725.

\bibitem{NishSieb06}
{\sc T.~Nishinou and B.~Siebert}, {\em {Toric degenerations of toric varieties
  and tropical curves}}, Duke Math. J., 135 (2006), pp.~1 -- 51.

\bibitem{OpreaPandharipande}
{\sc D.~Oprea and R.~Pandharipande}, {\em Quot schemes of curves and surfaces:
  virtual classes, integrals, {E}uler characteristics}, Geom. Topol., 25
  (2021), pp.~3425--3505.

\bibitem{PP17}
{\sc R.~Pandharipande and A.~Pixton}, {\em Gromov-{W}itten/{P}airs
  correspondence for the quintic 3-fold}, J. Amer. Math. Soc., 30 (2017),
  pp.~389--449.

\bibitem{PayFibers}
{\sc S.~Payne}, {\em Fibers of tropicalization}, Math. Z., 262 (2009),
  pp.~301--311.

\bibitem{RG71}
{\sc M.~Raynaud and L.~Gruson}, {\em Crit\`eres de platitude et de
  projectivit\'e. {T}echniques de ``platification'' d'un module}, Invent.
  Math., 13 (1971), pp.~1--89.

\bibitem{SU03}
{\sc P.~Sankaran and V.~Uma}, {\em Cohomology of toric bundles}, Comment. Math.
  Helv., 78 (2003), pp.~540--554.

\bibitem{Shah22}
{\sc A.~Shah}, {\em {$K$}-theoretic balancing conditions and the {G}rothendieck
  group of a toric variety}, J. Algebra, 611 (2022), pp.~175--210.

\bibitem{stacks-project}
{\sc T.~{Stacks Project Authors}}, {\em \textit{Stacks Project}}.
\newblock \url{https://stacks.math.columbia.edu}, 2018.

\bibitem{Sturmfels96}
{\sc B.~Sturmfels}, {\em Grobner bases and convex polytopes}, vol.~8, American
  Mathematical Soc., 1996.

\bibitem{tevelev2005compactifications}
{\sc J.~Tevelev}, {\em Compactifications of subvarieties of tori}, Amer. J.
  Math., 129 (2007), pp.~1087--1104.

\bibitem{TevNotes}
{\sc J.~Tevelev}, {\em {MSRI} lectures on tropical elimination theory}, 2009.

\bibitem{TV21}
{\sc J.~Tevelev and T.~Vogiannou}, {\em Spherical tropicalization}, Transform.
  Groups, 26 (2021), pp.~691--718.

\bibitem{ulirsch2014tropical}
{\sc M.~Ulirsch}, {\em Tropical compactification in log-regular varieties},
  Math. Z., 280 (2015), pp.~195--210.

\bibitem{Yu14}
{\sc T.~Y. Yu}, {\em The number of vertices of a tropical curve is bounded by
  its area}, Enseign. Math., 60 (2014), pp.~257--271.

\end{thebibliography}
 \end{document}